\documentclass[11pt,final]{amsart}

\usepackage{comment}
\usepackage{booktabs}       
\usepackage{amsfonts}       
\usepackage{nicefrac}       
\usepackage[margin=1in]{geometry}
\usepackage[noend]{algpseudocode}
\usepackage{amsmath,amsthm,amsfonts,amssymb,amscd}
\usepackage[foot]{amsaddr}
\usepackage{lastpage}
\usepackage{enumerate}
\usepackage{fancyhdr}
\usepackage{mathrsfs}
\usepackage[ruled,vlined]{algorithm2e}
\usepackage[table]{xcolor}
\usepackage{color}
\usepackage{graphicx}
\usepackage{listings}
\usepackage{hyperref}
\usepackage{bbm}
\usepackage{subfig}
\usepackage{caption}
\usepackage{diagbox}

\usepackage{tikz-qtree,tikz-qtree-compat}
\usepackage{tikz,pgf}
\usepackage{arydshln}

\usetikzlibrary{decorations,decorations.markings,decorations.text}
\usetikzlibrary{arrows.meta}
\usetikzlibrary{patterns}

\usetikzlibrary{positioning,patterns}
\usetikzlibrary{calc,fadings,decorations.pathreplacing,arrows,
datavisualization.formats.functions,shapes.geometric}

\date{}

\newtheorem{theorem}{Theorem}[section]

\def\boxit#1{\vbox{\hrule\hbox{\vrule\kern6pt
          \vbox{\kern6pt#1\kern6pt}\kern6pt\vrule}\hrule}}

\def\bse{\begin{eqnarray*}}
\def\ese{\end{eqnarray*}}
\def\be{\begin{eqnarray}}
\def\ee{\end{eqnarray}}
\def\bq{\begin{equation}}
\def\eq{\end{equation}}
\def\bse{\begin{eqnarray*}}
\def\ese{\end{eqnarray*}}

\def\T{^{\rm T}}

\newcommand{\corb}[1]{\textcolor{black}{#1}}

\newcommand{\bbN}{\mathbb{N}}

\newcommand{\bbS}{\mathbb{S}}

\newcommand{\bbC}{\mathbb{C}}

\newcommand{\bbP}{\mathbb{P}}

\newcommand{\bD}{\mathbf{D}}

\newcommand{\bU}{\mathbf{U}}

\newcommand{\bM}{\mathbf{M}}
\newcommand{\bN}{\mathbf{N}}

\newcommand{\bV}{\mathbf{V}}

\newcommand{\bx}{\mathbf{x}}
\newcommand{\by}{\mathbf{y}}

\newcommand{\bpsi}{\boldsymbol{\psi}}
\newcommand{\bphi}{\boldsymbol{\phi}}

\newcommand{\bSigma}{\boldsymbol{\Sigma}}

\newcommand{\0}{\mathbf{0}}

\newcommand{\mcB}{{\mathcal B}}
\newcommand{\mcC}{{\mathcal C}}
\newcommand{\mcD}{{\mathcal D}}
\newcommand{\mcE}{{\mathcal E}}
\newcommand{\mcF}{\mathcal{F}}

\newcommand{\mcK}{{\mathcal K}}

\newcommand{\mcN}{{\mathcal N}}
\newcommand{\mcO}{{\mathcal O}}
\newcommand{\mcP}{{\mathcal P}}

\newcommand{\mcQ}{{\mathcal Q}}

\newcommand{\mcT}{{\mathcal T}}

\newcommand{\eset}[1]{{\mathbb E} \left[ #1 \right] }

\def\R{\Bbb{R}}

\definecolor{darkgreen}{rgb}{0, 0.6, 0}
\definecolor{airforceblue}{rgb}{0.36, 0.54, 0.66}
\definecolor{applegreen}{rgb}{0.55, 0.71, 0.0}
\definecolor{asparagus}{rgb}{0.53, 0.66, 0.42}
\definecolor{cadetblue}{rgb}{0.37, 0.62, 0.63}
\definecolor{cambridgeblue}{rgb}{0.64, 0.76, 0.68}
\definecolor{olivine}{rgb}{0.6, 0.73, 0.45}
\definecolor{rufous}{rgb}{0.66, 0.11, 0.03}
\definecolor{sangria}{rgb}{0.57, 0.0, 0.04}

\newtheorem{defi}{Definition}

\newtheorem{prop}{Proposition}

\newtheorem{asum}{Assumption}

\theoremstyle{remark}
\newtheorem{exam}{Example}
\newtheorem{rem}{Remark}

\DeclareMathOperator{\spn}{span}

\begin{document}


 
\title{\textbf{Change Detection: A functional analysis perspective}}

 

    

    
    \author{Julio E. Castrill\'{o}n-Cand\'{a}s$^{\ddagger}$, Mark Kon$^{\ddagger}$}
    \email{jcandas@bu.edu, mkon@bu.edu}
 
 \address{
   ${\ddagger}$ Department of Mathematics and Statistics, 
  Boston University, Boston, MA 
  }
    
%

 
  
  





 \begin{abstract}
   We develop a new approach for detecting changes in the behavior of
   stochastic processes and random fields based on tensor product
   representations such as the Karhunen-Lo\`{e}ve expansion. From the
   associated
   eigenspaces of the covariance operator a series of nested function
   spaces are constructed, allowing detection of signals lying in 
   orthogonal subspaces. In particular this can succeed even if the
   stochastic behavior of the signal changes either in a global or
   local sense. A mathematical approach is developed to locate and
   measure sizes of extraneous components based on 
   construction of multilevel nested subspaces.  We show examples 
   in $\R$
   and on a spherical domain $\bbS^{2}$. However, the method is
   flexible, allowing the detection of orthogonal signals on general
   topologies, including spatio-temporal domains.
\end{abstract}

\maketitle


\section{Introduction}

Change detection is an important topic in statistics and has received
much attention, particularly in the context of time series and break
detection (see the literature review in
\cite{Aue2013,Jandyala2013}). There are many approaches to this
problem, including a posteriori change point analysis
\cite{Davis1995,Aue2009,Jirak2015}.  Other directions concentrate on
parameter changes \cite{Page1954,Hinkley1971,Moustakides1986,Lai1995}.
More recently, an avenue based on tracking changes in a linear model
was proposed in \cite{Chu1996} and extended in
\cite{Horvath2004,Fremdt2014,Aue2012,Wied2013,Kirch2018,Pape2016}. This
direction has been recently expanded to problems in $\R^{d}$ \cite{Dette2020} and combined with ideas involving self-normalization
\cite{Shao2010,Shao2015,Zhang2018}.


In this paper an orthogonal and new direction to change detection is developed, framed in the
context of functional analysis and tensor product representations such as
the Karhunen-Lo\`{e}ve \cite{Loeve1978} expansion. This method is very different from
the previous approaches --  Karhunen-Lo\`{e}ve (KL) expansions are an 
important method for representating stochastic processes and random fields, 
forming optimal tensor product representations. Due to the generality
of this approach, a large class of processes and fields can be
represented with high accuracy. Detection is achieved by constructing
nested subspaces adapted to eigenspaces of truncated
KL expansions.

In Section \ref{Introduction} the mathematical background is
discussed.  In particular, the KL expansion of a stochastic process is defined. In
Section \ref{section:mls} the theory of change detection via
application of nested function spaces is developed and applied
to stochastic process examples.  In Section \ref{CMLS} an algorithm for
the construction of these spaces is shown in detail.
This method is very general, allowing construction of  multilevel
bases on very general simplicial complex domains. An example application of this method to Spherical Fractional Brownian Motion
(SFBM) is shown in Section \ref{spherical}.

\section{Mathematical Background}
\label{Introduction}

The Karhunen Lo\`{e}ve expansion is
an important methodology that represents random fields in terms of
spatial-stochastic tensor expansions. It has been shown to be optimal 
in several ways, making it attractive for analysis of random fields.

Let $(\Omega,\mcF,\bbP)$ be a complete probability space, with $\Omega$ 
a set of outcomes, and $\mcF$ a $\sigma$-algebra of events equipped with
the probability measure $\bbP$.  Let $U$ be a domain of $\R^{d}$ and
$L^{2}(U)$ be the Hilbert space of all square integrable functions
$v:U \rightarrow \R$ equipped with the standard inner product
\[
\langle u,v \rangle = \int_{U} uv\, \mbox{d}\bx,
\]
for all $u(\bx),v(\bx) \in L^{2}(U)$.  In addition, let $L^{2}_{\bbP}(\Omega;
L^{2}(U))$ be the space of all functions $v:\Omega \rightarrow
L^{2}(U)$ equipped with the inner product
\[
\langle u,v \rangle_{L^{2}_{\bbP}(\Omega; L^{2}(U))} 
= \int_{\Omega} \langle u,v \rangle \, \mbox{d}\bbP,
\]
for all $u,v \in L^{2}_{\bbP}(\Omega;L^{2}(U))$.
\medskip

\begin{defi}~
\begin{itemize} 
\item Suppose $v \in L^{2}_{\bbP}(\Omega;L^{2}(U))$. Then $E_v :=
  \eset{v}$ is denoted as the mean of $v$.
\item For $v \in L^{2}_{\bbP}(\Omega;L^{2}(U))$ define the covariance function as
\[
{\rm Cov}(v(\bx,\omega),v(\by,\omega)) := \eset{ (v(\bx,\omega) -
  \eset{v(\bx,\omega)}) (v(\by,\omega) - \eset{v(\by,\omega)})}.
\]
\end{itemize}
\end{defi}
From the properties of Bochner integrals (see \cite{Light1985}) we
have that $E_v \in L^{2}(U)$ and that the covariance function
${\rm Cov}(\bx,\by) \in L^{2}(U \times U)$.  Define the operator $T:L^{2}(U)
\rightarrow L^2(U)$
\[
T(u)(\bx) := \int_{U}  {\rm Cov}(\bx,\by) u(\by)\,\mbox{d} \by
\]
for all $u \in L^{2}(U)$.  From Lemma 2 and Theorem 1 in
\cite{Harbrecht2016} there exists an orthonormal set of eigenfunctions
$\{\phi_k\}_{k \in \bbN}$, where $\phi_{k} \in L^{2}(U)$ and a
sequence of non-negative eigenvalues $\lambda_1 \geq \lambda_2 \geq
\dots$ such that $T \phi_k$ = $\lambda_k \phi_k$ for all $k \in \bbN$.

If $v \in L^{2}(\Omega;L^{2}(U))$ then from Proposition 2.8 in
\cite{Schwab2006} the random field $v$ can be represented in terms of
the \emph{Karhunen-Lo\`{e}ve} (KL) tensor product expansion
\begin{equation}
v(\bx,\omega) = E_v + \sum_{k \in \bbN} \lambda^{\frac{1}{2}}_{k} \phi_k(\bx) Y_{k}(\omega),
\end{equation}
where $\eset{Y_k Y_l} = \delta_{kl}$ and $\eset{Y_k} = 0$ for all $k,l \in \bbN$.

We now focus our attention on the truncated KL expansion, as this will be
important in the construction of change detection filters. For any $M
\in \bbN$ it is not hard to show that
\begin{equation}
\| v(\bx,\omega) - E_v - \sum_{k = 1}^{M} \lambda^{\frac{1}{2}}_{k} \phi_k(\bx) Y_{k}(\omega) \|_{L^{2}_{\bbP}(\Omega; L^{2}(U))}  
= \left( \sum_{k \geq M+1} \lambda_k \right)^{\frac{1}{2}}.
\label{intro:eqn2}
\end{equation}
Thus the decay of the eigenvalues controls the error of the
representation. Additionally, the truncated KL expansion has the property
of being optimal i.e. no other expansion of the same form is better in a sense to be specified.

To examine this further we note that from the definitions, $L^{2}_{\bbP}(\Omega;L^{2}(U))$ is isomorphic to the
tensor product space $L^{2}_{\bbP}(\Omega) \otimes L^{2}(U)$.  Suppose
that $H_M \subset L^{2}(U)$ is a finite dimensional subspace of
$L^{2}(U)$ such that $\dim H_M = M$ and $P_{H_M \otimes
  L^{2}_{\bbP}(\Omega)}: L^{2}(U) \otimes L^{2}_{\bbP}(\Omega)
\rightarrow H_M \otimes L^{2}_{\bbP}(\Omega)$ is an orthogonal
projection operator. Suppose $f \in L^{2}(U) \otimes
L^{2}_{\bbP}(\Omega)$, with $\eset{f} = 0$, then from Theorem 2.7 in
\cite{Schwab2006}
\[
\inf_{\begin{array}{c}
H_M \subset L^{2}(U) \\
\mbox{dim}\, S = M
\end{array}
} \|f - P_{H_M \otimes L^{2}_{\bbP}(\Omega)}f \|_{L^{2}_{\bbP}(\Omega) \otimes L^{2}(U)}
=
\left( \sum_{k \geq M+1} \lambda_k \right)^{\frac{1}{2}}
\]
where the infimum is achieved when $H_M = \mbox{span}\{\phi_1,\dots,\phi_{M}\}$.

The optimal expansion of any random field $v \in
L^{2}_{\bbP}(\Omega;L^{2}(U))$ will depend on the smoothness of 
$v$, which will have a direct impact on the decay of the eigenvalues
$\{ \lambda_k \}_{k=1}^{\infty}$.  Consider the Sobolev space
$H^{p}(U)$ with $p>0$, and its dual space $\tilde
H^{-p}(U)$. For functions $v \in L^{2}_{\bbP}(\Omega;H^{p}(U))$ (almost surely $p$-Sobolev smooth) we have as a consequence of Theorem 2 in \cite{Harbrecht2016}):

\begin{theorem}
If $v \in L^{2}_{\bbP}(\Omega;H^{p}(U))$, then the eigenvalues $\{
\lambda_k \}_{k=1}^{\infty}$ of the covariance operator $T:\tilde
H^{-p}(U) \rightarrow H^{p}(U)$ satisfy $\lambda_k \leq C k^{-2p}$ for
some constant $C>0$ independent of $k$ and $p$.
\label{intro:theo1}
\end{theorem}

\begin{exam} (\textit{Brownian Motion}) 
Suppose that $U = [0,1]$ and $W_t$ is the Wiener process with 
covariance function $\mbox{Cov}(t,s) = \min \{s,t\}$. 
The KL expansion of $W_t$ requires solving for eigenspaces defined by
\[
\int_{U}  \min \{s,t\} \phi_{k}(s)\,  ds = \lambda_k \phi_{k}(t). 
\]
For this type of stochastic process it is possible to analytically
solve for the eigenpair $(\lambda_k, \phi_k)$ for all $k \in \bbN$.
In \cite{Wang2008} it is shown that $\lambda_k = \frac{4}{(2k -
  1)^2 \pi^2}$, $\phi_k(t) = \sqrt{2} \sin(t/\sqrt{\lambda_k})$ and
$Y_{k}(\omega) \sim \mcN(0,1)$ i.i.d.  Thus we have
\[
W_t = \sqrt{2} \sum_{k \geq 1} \frac{2}{(2k -1)\pi} \sin((k - 1/2) \pi t) Y_k(\omega).
\]
\end{exam}

\begin{rem}
In general the KL expansion can be difficult to obtain. In particular, for non-Gaussian processes the
random coefficients $\{Y_{k}(\omega)\}_{k = 1}^{M}$ are not generally
independent. However, as will be shown in Section \ref{section:mls},
to build a change detection filter it is not necessary to explicitly
obtain the random coefficients $\{Y_{k}(\omega)\}_{k = 1}^{M}$. It is
sufficient only to characterize the eigenspaces from the decomposition $\{
(\lambda_k,\phi_k)\}_{k=1}^{M}$.  In practice, from a set of
realizations of $v(\bx,\omega) \in L^{2}_{\bbP}(\Omega;L^{2}(U))$ the
pair $(\lambda_k,\phi_k)$, for $k = 1,\dots,M$, can be estimated empirically using
the method of snapshots \cite{Castrillon2002}.
\end{rem}

\section{Multilevel orthogonal eigenspaces}
\label{section:mls}
Our main goal is to detect signals defined on the domain $U$ that do
not belong to the family of finite dimensional truncated KL expansions
\begin{equation}
v_M(\bx,\omega)  - E_v =  \sum_{k = 1}^{M} \lambda^{\frac{1}{2}}_{k} \phi_k(\bx) Y_{k}(\omega).
\end{equation}
To be more precise, we seek to detect signals that are orthogonal to
the eigenspace spanned by $\{\phi_1,\dots,\phi_M\}$ in a local and/or
global sense.


\begin{asum}
Without loss of generality assume that $E_v = 0$, and consider a
sequence of nested subspaces $V_0 \subset V_1 \dots \subset L^{2}(U)$
such that
\[
\overline{\bigcup_{k \in \bbN_{0}} V_{k}} = L^{2}(U) 
\]
and $V_{0} := \spn \{\phi_1,\phi_2, \dots,\phi_M \}$.  Furthermore,
let the subspaces $W_k \subset L^{2}(U)$,
for $k = 0,1,2,\dots$, be defined by $V_{k+1} = V_{k} \oplus W_{k}$ (all direct sums are orthogonal), so that
\[
\overline{ V_0 \bigoplus_{k \in \bbN_{0}} W_{k}} = L^{2}(U).
\]
\label{mls:assum1}
\end{asum}

\begin{prop}
For all $k \in \bbN_{0}$ and any function $ \psi \in W_{k}$ we have
\[
\int_U \phi_{l} \psi\,\emph{d} \bx = 0
\]
for $l = 1,\dots,M$.
\label{mls:prop1}
\end{prop}
\begin{proof}
Since $V_{k} = V_{k-1} \oplus W_{k-1}$, it follows $W_{k} \perp V_{0}$.
\end{proof}

\medskip

Suppose $v \in L^{2}_{\bbP}(\Omega;L^{2}(U))$.  Then the projection of
$v$ onto the multilevel spaces $\{W_{k}\}_{k \in \bbN}$ characterizes
the signal in terms of components orthogonal to the eigenspace
$V_{0}$. Given that computational power is limited, we seek to
construct the space $V_{0}$ from the eigenfunctions $\phi_1, \phi_2,
\dots, \phi_M$ so that
\[
\|v(\bx,\omega) - v_{M}(\bx,\omega)
\|_{L^{2}_{\bbP}(\Omega; L^{2}(U))}  
= \left( \sum_{k \geq M+1} \lambda_k \right)^{\frac{1}{2}} \leq \mbox{\tt tol}
\]
for a desired tolerance $\mbox{\tt tol}>0$. The choice of $M \in \bbN$ has a
direct impact on the magnitude of the projection of $v$ onto the
sum of the remainder spaces $\{W_{k}\}_{k \in \bbN_0}$. We shall now study the
effect of truncating the KL expansion of $v \in
L^{2}_{\bbP}(\Omega;L^{2}(U))$ onto the above projections.

\begin{rem}
The following discussion is also applicable to other non-KL expansions
of random fields of the same form. 
%
The coefficients $\lambda_k$ and
functions $\phi_k$ do not necessarily need to be restricted to the eigenvalue
decomposition of the covariance operator $T:L^{2}(U) \rightarrow
L^{2}(U)$.  However, for simplicity of the exposition, 
for non-KL expansions we use the same notation and it is still assumed that $\lambda_1 \geq
\lambda_2 \geq \dots > 0$ and $\phi_1,\phi_2,\dots$ form an orthogonal
set. In the rest of the paper we assume KL expansion unless otherwise
noted.
\end{rem}


\begin{asum}
For all $l \in \bbN_0$ let $\{ \{ \psi^{l}_{k} \}_{k = 1}^{M_l} \}_{l
  \in \bbN_{0}}$ be a collection of orthonormal functions with
$W_{l} = \spn \{\psi_1^l, \dots, \psi_{M_{l}}^l \}$ and $M_l := \dim
W_l$.
\label{mls:assum2}
\end{asum}
Since the $L^{2}(U)$ basis $\{ \{ \psi^{l}_{k} \}_{k = 1}^{M_l} \}_{l
  \in \bbN}$ is orthonormal, for any function $g \in L^{2}(U)$ the
orthogonal projection coefficient onto the function $\psi^{l}_{k} \in
W_{l}$ is 
\begin{equation}
d^l_k := \int_U g \psi^l_k \,\mbox{d} \bx.    
\label{mls:eqn0}
\end{equation}

We will now study the effect of the truncation parameter $M$ on the
projection coefficients on the spaces $W_{l}$ for $l \in \bbN_{0}$.


\begin{theorem}
Suppose that $v \in L^{2}_{\bbP}(\Omega;L^{2}(U))$ with KL expansion
\[
v(\bx,\omega) = 
\sum_{p \in \bbN} \lambda^{\frac{1}{2}}_{p} \phi_p(\bx) Y_{p}(\omega).
\]
Then for all $l \in \bbN_0$, $k = 1,\dots, M_l$ and projection 
coefficients 
\[
d^l_k(\omega) = \int_U v(\bx,\omega) \psi^l_k \,\mbox{d} \bx
\]
a.s. then 
\[
\eset{d^{l}_k} =  0 
\ \ \,\mbox{\rm and}\,\,\ \ \ 
\eset{(d^{l}_k)^2} \leq \sum_{i\geq M+1} \lambda_{i}.
\]
\end{theorem}
\begin{proof} $\eset{d^{l}_k} =  0$ follows trivially from
$\eset{Y_{i}} = 0$ for all $i \in \bbN_0$.
From Proposition \ref{mls:prop1} and $W_{k} \perp V_{0}$ we have that
\[
\begin{split}
\eset{(d^{l}_k)^2}
&= \eset{\left(\int_U v \psi^l_k \,\mbox{d} \bx \right)^{2}} = 
\eset{\left( \int_U \sum_{i\geq M+1} \lambda^{\frac{1}{2}}_{i} \phi_i(\bx)\psi^l_k(\bx)Y_{i}(\omega)  \,\mbox{d} \bx \right)^{2}} \\
& = 
\eset{
\sum_{i\geq M+1}  \sum_{j \geq M+1} 
Y_{i}(\omega)  Y_{j}(\omega)  \lambda^{\frac{1}{2}}_{i}  \lambda^{\frac{1}{2}}_{j}  
\int_U \phi_i(\bx)\psi^l_k(\bx)\,\mbox{d} \bx
\int_U \phi_j(\bx)\psi^l_k(\bx)\,\mbox{d} \bx}
\end{split}
\]
From the  property that  $\eset{Y_i Y_j} = \delta_{ij}$
($i,j \in \bbN_0$),  we have that
\[
\eset{(d^{l}_k)^2} = 
\sum_{i\geq M+1}  
\lambda_{i}
\int_U \phi_i(\bx)\psi^l_k(\bx)\,\mbox{d} \bx
\int_U \phi_i(\bx)\psi^l_k(\bx)\,\mbox{d} \bx. \\
\]
The conclusion follows from Cauchy-Schwartz and the orthonormality of
the basis $\{ \{ \psi^{l}_{k} \}_{k = 1}^{M_l} \}_{l = 0}^{\infty}$.
\end{proof}

\begin{rem}
As $M$ increases not only the approximation error of the KL expansion
is reduced and dominated by the sum of eigenvalues, but the variance
of the coefficients $d^l_k$ is also controlled by the same quantity.
Thus by using the Chebyshev inequality the projection coefficients on
$W_l$ converge in probability to 0 if $\left( \sum_{i \geq M+1}
\lambda_i \right) \rightarrow 0$. The significance of this theorem is t
hat each of the projection coefficients
becomes more deterministic as the truncation parameter $M$ is increased.
\end{rem}

\subsection{Global detector}

Suppose that $d^l_k$ are the projection coefficients of a novel signal
$u(\bx,\omega) \in L^{2}_{\bbP}(\Omega;L^{2}(U))$.
The coefficients $d^l_k$ provide a mechanism to detect the magnitude
of the novel part of the signal
that is orthogonal to the eigenspace
$V_0$. In more colloquial terms, we desire to detect the components of
$u(\bx,\omega)$ with stochastic properties that are different from
the eigenspace.

Suppose that $u(\bx,\omega) = v_{M}(\bx,\omega) + w(\bx,\omega)$
i.e. the signal $u(\bx,\omega)$ is formed from the components
$v_{M}(\bx,\omega) \in V_0$ and $w(\bx,\omega) \in V_0^{\perp}$ (See
Figure \ref{mls:fig1}). The goal then is to detect the orthogonal
component $w(\bx,\omega)$ that does not belong in the eigenspace
$V_0$.


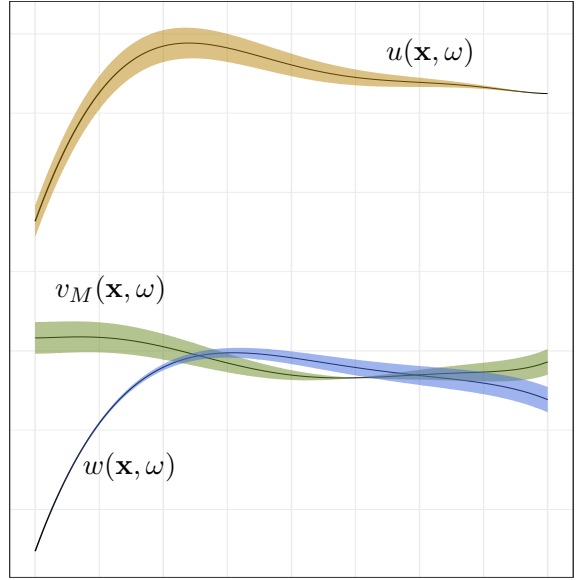
\begin{figure}[htb]
    \centering
    \begin{tikzpicture}
    \node at (0,0)
    {
        \includegraphics[width = 8cm, height = 8cm]{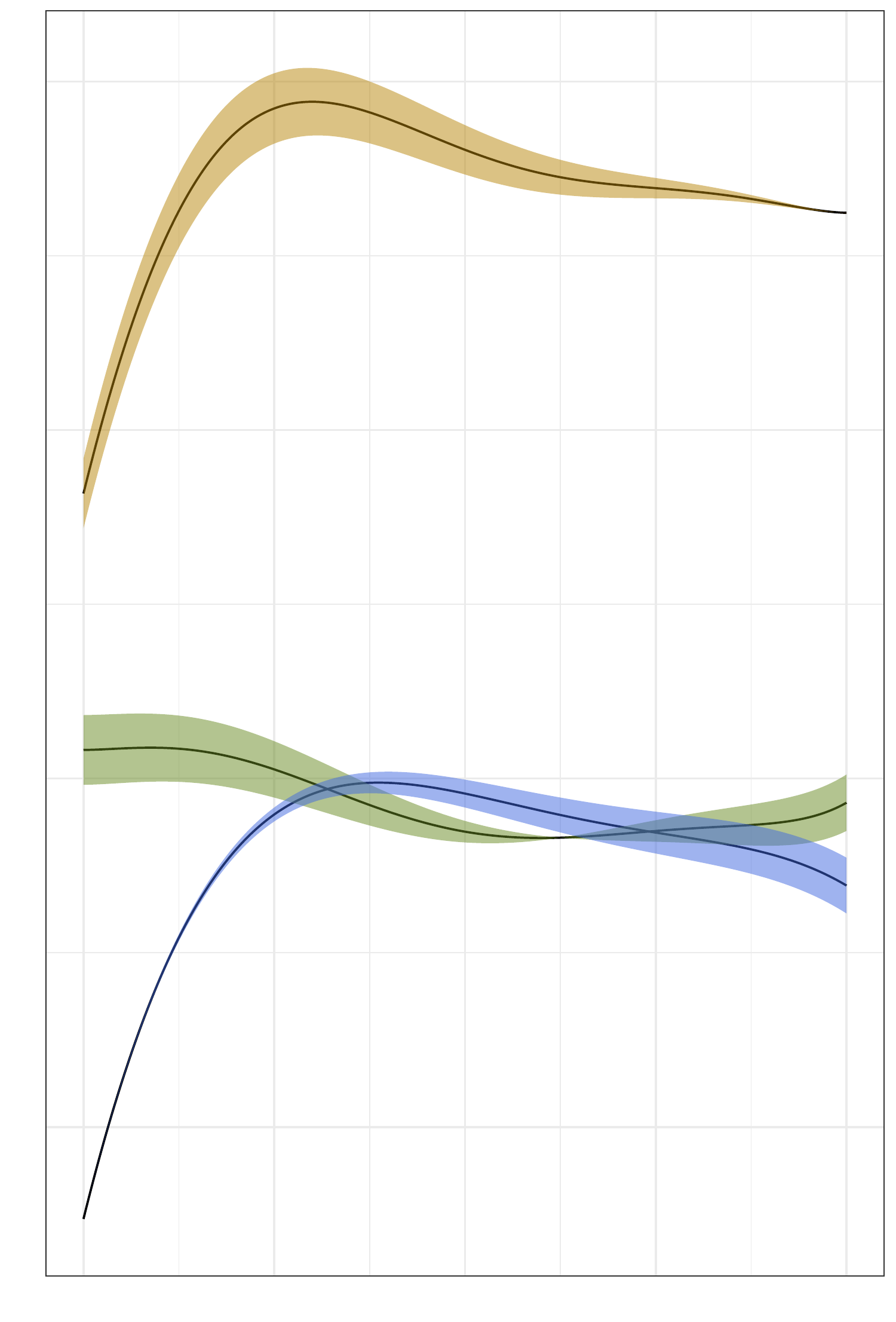}
    };
    
    \node at (2,3.25) {$u(\bx,\omega)$};
    \node at (-2.25,0.10) {$v_{M}(\bx,\omega)$};
    \node at (-2.0,-2.25) {$w(\bx,\omega)$};

    \end{tikzpicture}
    \caption{Decomposition of the signal $u(\bx,\omega)$ into the two
      orthogonal components $v_M(\bx,\omega)$ and
      $w(\bx,\omega)$. Given the known random field $v_M(\bx,\omega)
      \in V_0$ a.s. in $\Omega$ the objective is to detect the
      orthogonal component $w(\bx,\omega) \in V_0^{\perp}$ a.s..}
    \label{mls:fig1}
\end{figure}

In the following theorems it is assumed that $d^{l}_k$ are defined
as in equation \eqref{mls:eqn0}.
\begin{theorem}
Suppose that $u(\bx,\omega) = v_{M}(\bx,\omega) + w(\bx,\omega)$ for
some $w(\bx,\omega) \in L^{2}_{\bbP}(\Omega;L^{2}(U)) $ and
$w(\bx,\omega) \perp V_{0}$ almost surely.  Then, almost surely,
\[
\sum_{l \in \bbN_0} 
\sum_{k=1}^{M_l}
(d^{l}_{k})^{2} = 
\| w(\bx,\omega)\|^{2}_{L^{2}(U)}
\]
and
\[
\sum_{l \in \bbN_0} 
\sum_{k=1}^{M_l}
\eset{(d^{l}_{k})^{2}} = 
\| w(\bx,\omega)\|^{2}_{ L^{2}_{\bbP}(\Omega;L^{2}(U)) }.
\]
\end{theorem}
\begin{proof}
This is immediate from the fact that $\{ \{ \psi^{l}_{k} \}_{k =
  1}^{M_l} \}_{l = 0}^{\infty}$ are orthonormal.
\end{proof}

The implication of the previous theorem is that the orthogonal signal
$w(\bx,\omega)$ can be determined exactly from the projection
coefficients in the space $V_0^{\perp}$.

We now study the effect of replacing the finite dimensional signal
$v_M(\bx,\omega)$ with the full KL expansion of the signal. In many
cases the infinite dimensional signal will be provide a more informative and useful model, including 
for
KL expansions of Gaussian random fields.  Nevertheless, in practical situations computational limitations among others will limit the
dimensionality of the eigenspace $V_{0}$ to a finite
level $M$. The detector coefficients $d^l_k$ will be affected by
the magnitude of the truncation of the KL expansion. However, as $M$
increases the error due to the truncation will rapidly decay as the
sum of the remaining eigenvalues.

\begin{theorem}
Let $t_{M} := \sum_{j\geq M +1} \lambda_j$ and suppose that
$u(\bx,\omega) = v(\bx,\omega) + w(\bx,\omega) $ for some
$w(\bx,\omega) \in L^{2}_{\bbP}(\Omega;L^{2}(U))$, with $w(\bx,\omega)
\perp V_{0}$ almost surely.  Then
\[
\| w(\bx,\omega)\|^{2}_{ L^{2}_{\bbP}(\Omega;L^{2}(U)) } ( 1 - 2
t_M
) 
+
t_M
\leq
\sum_{l \in \bbN_0} 
\sum_{k=1}^{M_l}
\eset{(d^{l}_{k})^{2}} 
\leq 
\| w(\bx,\omega)\|^{2}_{ L^{2}_{\bbP}(\Omega;L^{2}(U)) }
( 1 + 2
t_M
) 
+
t_M.
\]
\label{mls:theo3}
\end{theorem}
\begin{proof}
Let $P:L^{2}(U)\rightarrow V_0^{\perp}$ be the orthogonal
projection. 
Since 
\[
u(\bx,\omega) = v_{M}(\bx,\omega) + 
\sum_{p \geq M+1} \lambda^{\frac{1}{2}}_{p}
\phi_p(\bx) Y_{p}(\omega)
+
w(\bx,\omega),
\]
it follows
\[
Pu(\bx,\omega) = 
\sum_{p \geq M+1} \lambda^{\frac{1}{2}}_{p} \phi_p(\bx) Y_{p}(\omega)
+
w(\bx,\omega).
\]
Given that $\{ \{ \psi^{l}_{k} \}_{k = 1}^{M_l} \}_{l = 0}^{\infty}$ 
forms an orthonormal basis for $V_0^{\perp}$, we have
\begin{equation}
\begin{split}
\sum_{l \in \bbN_0} 
\sum_{k=1}^{M_l}
\eset{d^{l}_{k}}^{2} 
& = 
\| Pu(\bx,\omega)\|^{2}_{ L^{2}_{\bbP}(\Omega;L^{2}(U)) }
=
\| \sum_{p \geq M + 1} \lambda^{\frac{1}{2}}_{p} \phi_p(\bx) Y_{p}(\omega)
+ w(\bx,\omega)\|^{2}_{ L^{2}_{\bbP}(\Omega;L^{2}(U)) } \\
&=
\| \sum_{p \geq M+1} \lambda^{\frac{1}{2}}_{p} \phi_p(\bx) Y_{p}(\omega)
\|^{2}_{ L^{2}_{\bbP}(\Omega;L^{2}(U)) }
+
\| 
w(\bx,\omega)\|^{2}_{ L^{2}_{\bbP}(\Omega;L^{2}(U)) } \\
&+
2 
\eset{\int_U \Big(\sum_{p \geq M+1} \lambda^{\frac{1}{2}}_{p} \phi_p(\bx) Y_{p}(\omega)
\Big) w(\bx,\omega)\, \mbox{d}\bx},
\end{split}
\label{mls:eqn1}
\end{equation}

\begin{equation}
\| \sum_{p \geq M + 1} \lambda^{\frac{1}{2}}_{p} \phi_p(\bx) Y_{p}(\omega)
\|^{2}_{L^{2}_{\bbP}(\Omega;L^{2}(U)) } 
= t_M.
\label{mls:eqn2}
\end{equation}
From Cauchy-Schwartz (both with respect to the probability and the
Lesbegue measure) we have
\begin{equation}
\begin{split}
\left| \eset{ \int_U 
\phi_p(\bx) 
Y_{p}(\omega)
w(\bx,\omega)\, \mbox{d}\bx }
\right|
&\leq
\int_U 
\| \phi_p(\bx) 
Y_{p}(\omega)\|_{L^{2}_{\bbP}(\Omega)}
\| w(\bx,\omega)
\|_{L^{2}_{\bbP}(\Omega)}
\, \mbox{d}\bx \\
&=
\int_U  |\phi_p(\bx)|
\| w(\bx,\omega)
\|_{L^{2}_{\bbP}(\Omega)} \mbox{d}\bx \\
&\leq \|\phi_p(\bx)\|_{L^{2}(U)} 
\| w(\bx,\omega)
\|_{L^{2}_{\bbP}(\Omega;L^{2}(U)) }  \\
&=
\| w(\bx,\omega)
\|_{L^{2}_{\bbP}(\Omega;L^{2}(U)) }. 
\end{split}
\label{mls:eqn3}
\end{equation}
Inserting equations \eqref{mls:eqn3} and \eqref{mls:eqn2} in
\eqref{mls:eqn1} we reach the conclusion.
\end{proof}

\begin{rem}
The previous theorem provided a mechanism to determine the
intensity of the orthogonal signal $w(\bx,\omega)$ given
the size of the truncation parameter $M$.  Thus the larger $M$
is, depending on the decay of $\lambda_p$, $p = M+1,\dots$, the
size of the coefficients $d^l_k$ can be used to determine more
precisely the size of the perturbation $w(\bx,\omega)$ both in the
local and global sense.
\end{rem}

\begin{theorem}
Suppose that a signal $u(\bx,\omega) = 
v_{M}(\bx,\omega) + w(\bx,\omega)
$  for $v_M\in V_0$, and for some $w(\bx,\omega) \in L^{2}_{\bbP}(\Omega;L^{2}(U))$
then
\[
Pu(\bx,\omega)=Pw(\bx,\omega) = \sum_{l \in \bbN_0} 
\sum_{k=1}^{M_l}
d^{l}_{k}(\omega) \psi^{l}_{k}(\bx)  
\,\emph{and}\,
\sum_{l \in \bbN_0} 
\sum_{k=1}^{M_l}
\eset{(d^{l}_{k})^{2}} 
=
\| Pw(\bx,\omega)\|^{2}_{ L^{2}_{\bbP}(\Omega;L^{2}(U)) },
\]
where $P:L^{2}(U)\rightarrow V_0^{\perp}$ is the orthogonal
projection. 
\end{theorem}
\begin{proof}
Immediate
\end{proof}

\begin{rem}
The sharpness of the bound in Theorem \ref{mls:theo3} depends
on the decay of the coefficients $\lambda_{M+1}, \lambda_{M+2}, \dots$.
From Theorem \ref{intro:theo1} the decay of the eigenvalues is
related to the smoothness of the realization $v(\cdot,\omega) \in H^{p}(U)$, where $p \in \bbN_0$, from which we have $t_M \leq C \sum_{k=M+1}^{\infty}
k^{-2p}$, with $C$ independent of $k$ and $p$. For $p>0$ we have that
$t_M \approx (M+1)^{-2p}$.
\end{rem}

\subsection{Global - local detector}

Suppose that $\{\chi^l_k\}_{k \in \mcK(l)}$ is a finite disjoint
partition at level $l$ of the domain $U$ such that
\[
U = \bigcup_{k \in \mcK(l)} \chi^l_k,
\]
where $\mcK(l)$ is an index set corresponding to each element in the
the partition. For each $l \in \bbN_{0}$ let $U^l := \{ \chi^l_k \}_{k
  \in \mcK(l)}$ and assume that $U^{l+1}$ is refinement of $U^{l}$.
We make the assumption that the support of any basis function
$\psi^{l}_{k} \in W_l$ is given by a union of sets in $U^l$. In
Section \ref{CMLS} the construction of the finite dimensional spaces
$V_0, \dots V_n$ and $W_0, \dots W_{n-1}$ with compactly supported
basis functions will be described in detail.

Suppose that $\tilde U \subset U$ and let
\[
\mcC (\tilde U) := \{(k,l) ;\,\, l \in \bbN_0,k \in \mcK(l),
\,\mbox{supp}\, 
\tilde U \cap \mbox{supp}\, \psi^l_k \neq \emptyset  
\},
\]
where by $\mbox{supp}\,\tilde U$ we mean the union of all sets in $\tilde U$.
It is clear that any projection coefficient $d^l_k = 0$ (see eq.  (\ref{mls:eqn0})) if $(l,k)
\notin \mcC(\tilde U)$. Thus the coefficients that correspond to the
intersection of the domain $\tilde U$ and the support of the basis functions $\{
\{ \psi^{l}_{k} \}_{k = 1}^{M_l} \}_{l \in \bbN_0}$ are sufficient to
detect any changes on $\tilde U$. The following example
explores this property in more detail.

\begin{exam}
Let $v_M(x,\omega) 
= 1 + Y_{1}(\omega)  \left(
    \frac{\sqrt{\pi}L}{2} \right)^{\frac{1}{2}}
    + \sum_{k = 2}^{M}
  \lambda_{k}^{\frac{1}{2}}
  \phi_{k}(x)Y_{k}(\omega)$
  be a stochastic process defined 
  on the domain  $[0,\tau]$ 
where  
\[
\phi_{k}(x) : = \left\{
\begin{array}{cc}
  \sin \left( \frac{\lfloor \frac{k}{2} \rfloor \pi x}{L_{p}} \right) &
  \mbox{if $k$ is even}\\ \cos \left( \frac{\lfloor \frac{k}{2} \rfloor \pi
      x}{L_{p}} \right) & \mbox{if $k$ is odd}\\
\end{array}
\right. 
\]
are orthogonal and 
\[
\sqrt{\lambda_k} := (\sqrt{\pi} L)^{\frac{1}{2}} \exp \left( 
-\frac{(\lfloor \frac{k}{2}
\rfloor \pi L )^{2}}{8} 
\right).
\]
The random variables $Y_1, \dots, Y_M$ are assumed to be uniform in
$[-\sqrt{3},\sqrt{3}]$ and independently identically distributed.  In
\cite{nobile_tempone_08} the authors show that $v_M(x,\omega)$ is the
truncation of the infinite dimensional random field $v$ with the
covariance:
\begin{equation}
\mbox{Cov}(v(x,\omega), v(y,\omega)) =  \exp \left( - \frac{(x - y)^2}{L^2_c}
\right),
\label{mls:eqn4}
\end{equation}
where $L_p = \max \{ \tau, 2L_c \} $ is the length correlation and $L
= L_c / L_p$.  Suppose that $\tau = 1$ and $L_c = 0.01$, so that $L_p = 1$
and $L = 0.01$. Furthermore, let $u(x,\omega) = v_M(x,\omega) + w(x)$
where $\tilde U = [0.3,0.7]$, $w(x) = 1_{\tilde U} 0.05 \exp \left( -
\frac{(x - x_s)^2}{\sigma^2} \right)$, $x_s = 0.5$ and $\sigma =
10^{-3/2}$.  The compactly supported \emph{multilevel basis} (MB) for $V_{0}
\oplus W_0 \oplus W_1 \oplus \dots$ is constructed such that $V_{0}$
is the span of $\{\phi_1,\dots,\phi_M\}$. The objective now is to
detect the presence of the smooth Gaussian function $w(x)$ given the
signal $u(x,\omega)$.

In Figure \ref{mls:fig2} (a) the signals $u(x,\omega)$,
$v_M(x,\omega)$ (solid and dashed lines with left vertical axis) and
$w(x)$ (solid orange line with right vertical axis) are plotted.  The
deformation $w(x)$ is added to the baseline stochastic process
$v_M(x,\omega)$ to obtain $u(x,\omega)$.  In Figure \ref{mls:fig2}
(b), (c) and (d) the projection coefficients of the signal
$u(x,\omega)$ onto the multilevel spaces $W_{n-1}$, $W_{n-2}$ and
$W_{n-3}$, for $n = 6$, are plotted. Notice that around the center of
$w(x)$ at $x_s = 0.5$ the projection coefficients are clearly
non-zero. Thus the local presence of the non-zero coefficients
indicated that around $x_s = 0.5$ the usual behavior of the signal
$v_M(x,\omega)$ changes. Furthermore, from Theorem \ref{mls:theo3}
we
can conclude that
\[
Pw(\bx) = \sum_{(k,l) \in \mcC(\tilde U)} 
d^{l}_{k} \psi^{l}_{k}(x)  
\,\mbox{and}\,
\sum_{(k,l) \in \mcC(\tilde U)} 
(d^{l}_{k})^{2}
=
\| Pw(x)\|^{2}_{L^{2}(U)},
\]
where $P$ is as in Theorem 3.4.
\label{mls:example1}
\end{exam}
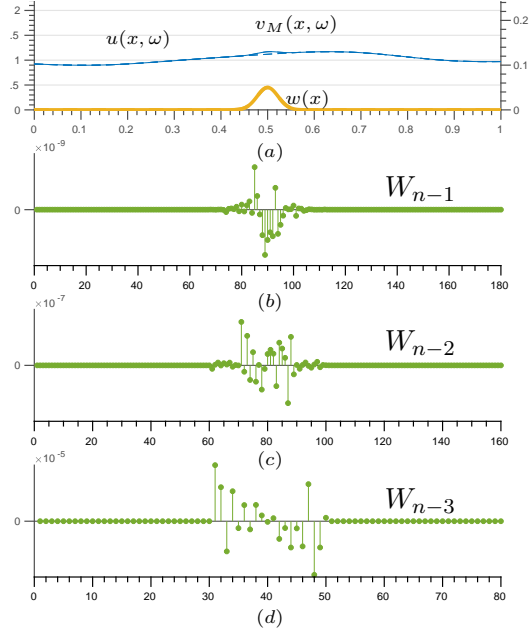
\begin{figure}
    \centering

    \begin{tikzpicture}
    \node at (0,0)
    {
        \includegraphics[scale = 0.45,
        trim = 3.5cm 7.4cm 2.75cm 2.75cm,
        clip=true]{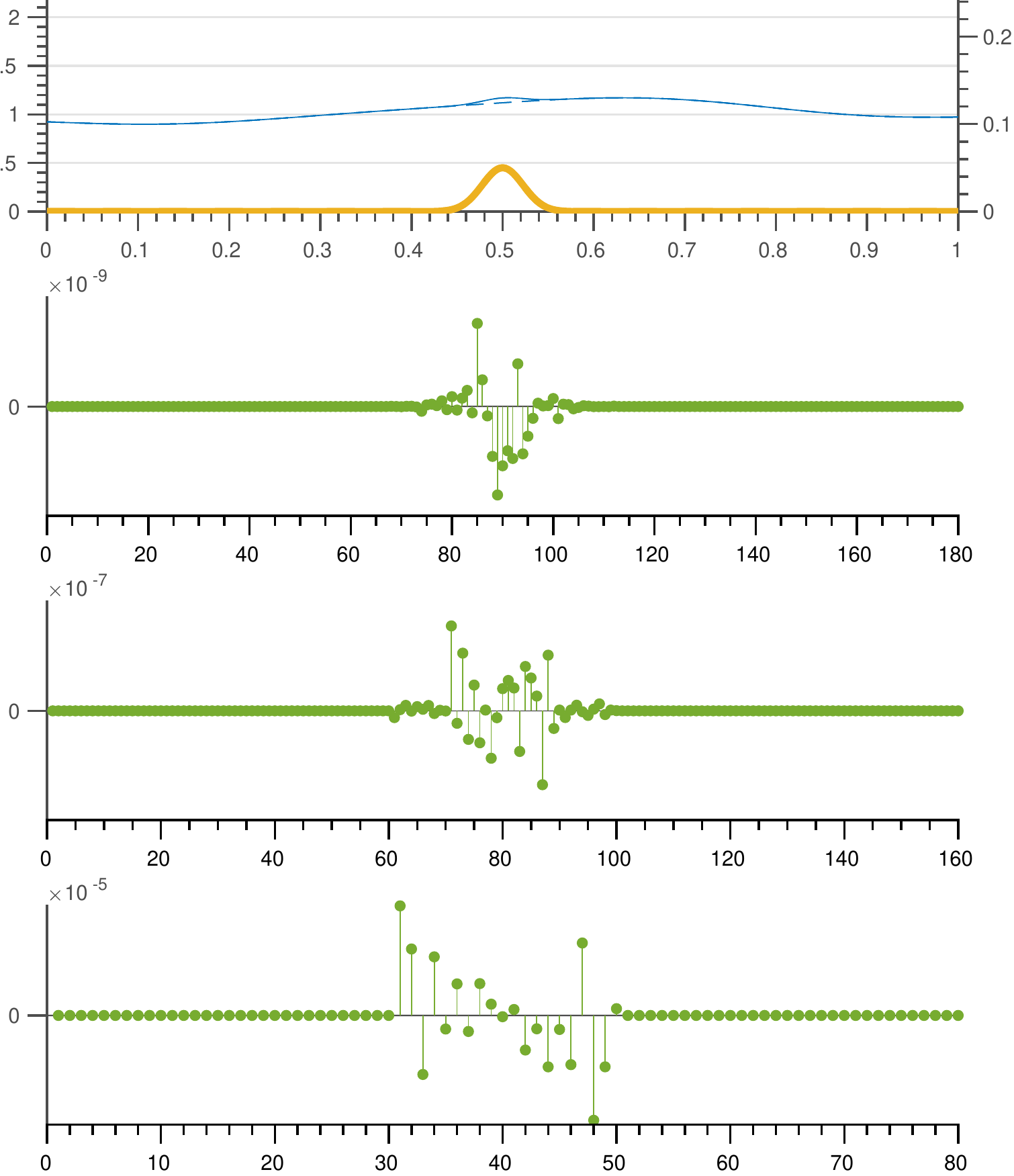}
    };

    \node at (0,2) {\tiny $(a)$};
    \node at (0,0) {\tiny $(b)$};
    \node at (0,-2.1) {\tiny $(c)$};
    \node at (0,-4.2) {\tiny $(d)$};

    \node at (-1.75,3.5) {\tiny $u(x,\omega)$};  
    \node at (0.35,3.7) {\tiny $v_{M}(x,\omega)$};  
    \node at (0.5,2.7) {\tiny $w(x)$};
    
    \node at (2,1.5) {$W_{n-1}$};
    \node at (2,-0.55) {$W_{n-2}$};
    \node at (2,-2.65) {$W_{n-3}$};
    
    \end{tikzpicture}
    
    \caption{Example of the functional analysis approach to change
      detection. (a) The original signal $v_M(x,\omega)$ (dash line
      with left vertical axis) with the Gaussian bump $w(x)$ (solid
      orange line with right vertical axis) with support on the domain
      $\tilde U := [0.3,0.7]$. Adding these two signals give
      $u(x,\omega)$ (solid line with left vertical axis).  The
      multilevel basis is constructed such that $W_k$, $k \in \bbN_0$,
      is orthogonal to the signal $v_M(x,\omega)$.  (b), (c) and (d)
      Detection of $w(x)$ at levels $W_{n-1}$, $W_{n-2}$ and
      $W_{n-3}$, where $n = 6$.  }
    \label{mls:fig2}
\end{figure}
\begin{rem}
In practice we assume that $u(\bx,\omega)$, $v_M(\bx,\omega)$ and
$w(\bx,\omega)$ belong in the finite dimensional space $V_n = V_0
\oplus W_0 \oplus \dots \oplus W_{n-1}$ for some finite fixed $n \in
\bbN_0$. For example, $V_n$ can be the span of disjoint characteristic
functions (Haar basis) over the domain $U$.  This is a reasonable
assumption since data are collected as samples. In Example
\ref{mls:example1} the signal is formed from 500 equally spaced
samples of $v_M(x,\omega)$ from $[0,1]$, and we have chosen $n = 6$ in $V_n$ above.
\end{rem}

\begin{exam} The multilevel approach can also be applied to non-KL expansions,
such as the following.
Let $v_M(x,\omega) = 1 + Y_{1}(\omega) \left( \frac{\sqrt{\pi}L}{2}
\right)^{1/2} + \sum_{k = 2}^{M} \sqrt{\mu_{k}}
\phi_{k}(x)Y_{k}(\omega)$ be a stochastic process, where $x \in
    [0,1]$, with $\mu_k$ and $\phi_k(x)$ defined as in Example
    \ref{mls:example1}. However, we fix $L_p = 1/4$ and $L = 1/4$. Let
    $\tilde U = [0.3,0.7]$ with $w(x) = 1_{\tilde U} 0.5 \exp \left( -
    \frac{(x - x_s)^2}{\sigma^2} \right)$, $x_s = 0.5$ and $\sigma =
    10^{-3/2}$. Note that this example does not necessarily have the
    covariance structure as shown in Equation \eqref{mls:eqn4}.
    However, these coefficients lead to a more oscillatory structure
    for $v_{M}(x,\omega)$. It is hard for the observer to distinguish
    $w(x)$ in $u(x,\omega)$ from $v_{M}(x,\omega)$ without knowledge
    of the location.  By building the multilevel spaces adapted to
    $v_M(x,\omega)$ the filter coefficients for levels $W_{n-1}$,
    $W_{n-2}$ and $W_{n-3}$ easily detect the location of $w(x)$ (See
    Figure \ref{mls:fig3}).
 \label{mls:example2}
\end{exam}

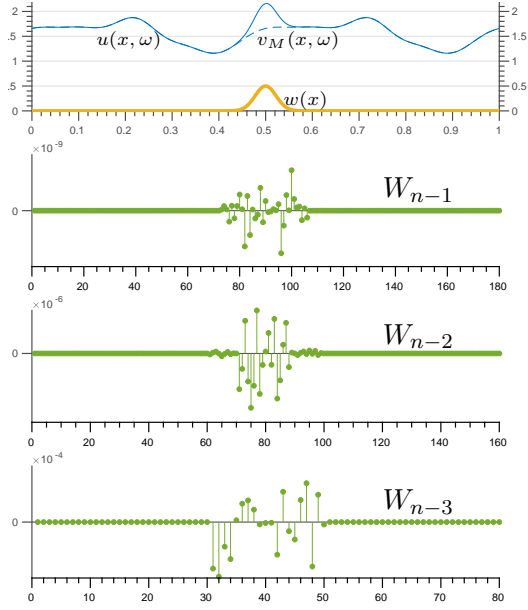
\begin{figure}
    \centering

            \begin{tikzpicture}
    \node at (0,0)
    {
        \includegraphics[scale = 0.45,
        trim = 3.5cm 7.4cm 2.75cm 2.75cm,
        clip=true]{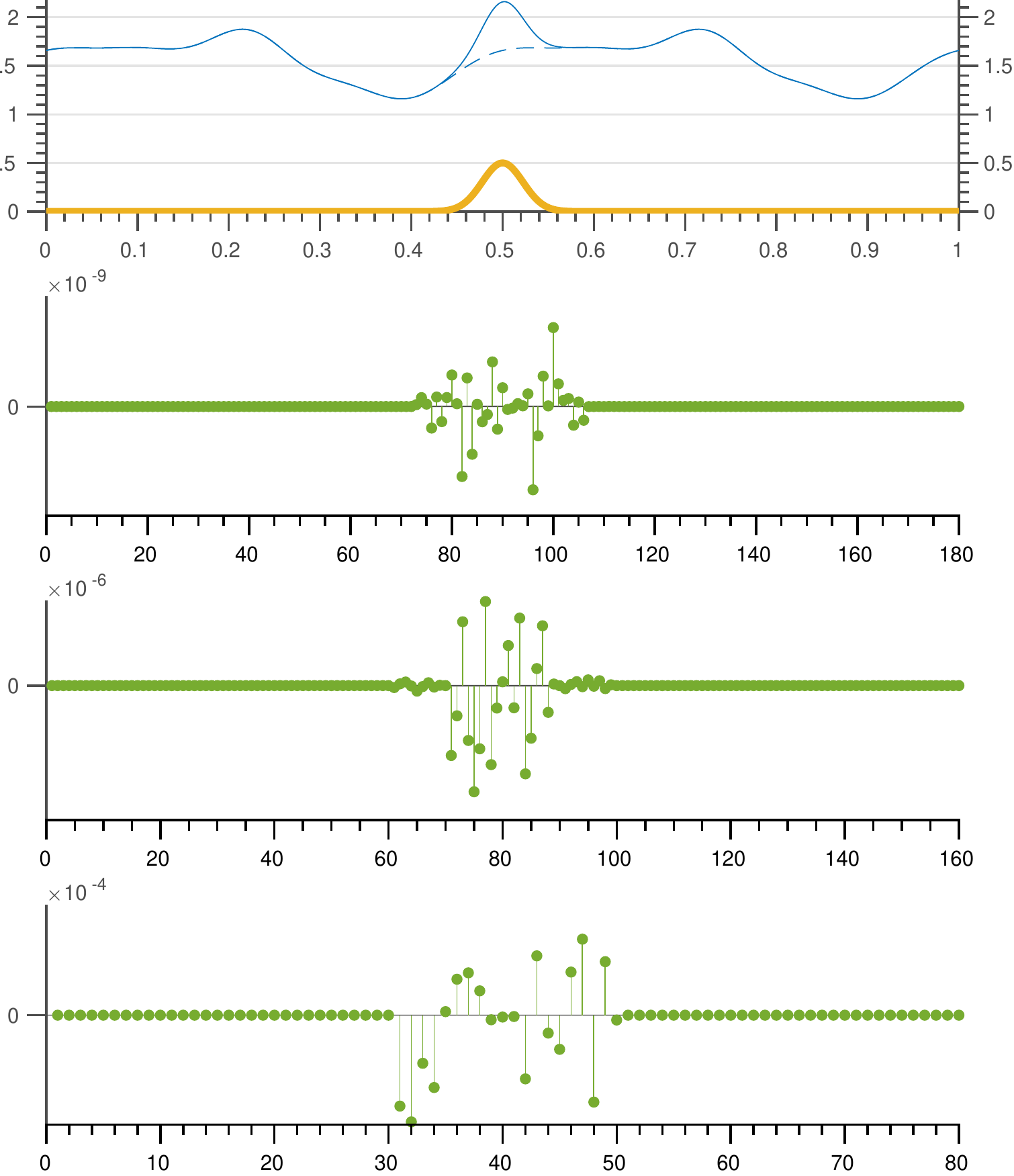}
    };
    
    \node at (-1.85,3.5) {\tiny $u(x,\omega)$};  
    \node at (0.4,3.5) {\tiny $v_{M}(x,\omega)$};  
    \node at (0.5,2.7) {\tiny $w(x)$};
    
    \node at (2,1.5) {$W_{n-1}$};
    \node at (2,-0.45) {$W_{n-2}$};
    \node at (2,-2.65) {$W_{n-3}$};
    
    \end{tikzpicture}

    \caption{Example of change detection for $v_{M}(x,\omega)$ with $L
      = L_p = 1/4$.  It is hard for the observer to distinguish $w(x)$
      in $u(x,\omega)$ from $v_{M}(x,\omega)$ without knowledge of the
      location.  By building the multilevel spaces adapted to
      $v_M(x,\omega)$ the filter coefficients for levels $W_{n-1}$,
      $W_{n-2}$ and $W_{n-3}$ easily detect the location of $w(x)$.  }
    \label{mls:fig3}
\end{figure}

\section{Construction of multilevel orthogonal eigenspace}
\label{CMLS}

For many practical problems the domain $U$ will be restricted to some
form of a mesh.  The multilevel basis (MB) of the finite dimensional spaces $V_n = V_0
\oplus W_0 \oplus \dots \oplus W_{n-1}$ for a finite fixed $n \in
\bbN_0$ can be constructed on this mesh. The construction of the MB
for problems in $\R^{3}$ within the context of polynomials and integral
operators was first proposed in \cite{Tausch2001}. In
\cite{Castrillon2016a} this was modified for discrete domains in the
context of Kriging  and high dimensional problems from
the work in \cite{Castrillon2020}.

\begin{defi}~
Suppose that $\mcT$ is a collection of $N$ simplices in $\R^{d}$.
Then $\mcT$ is a $k$-simplicial complex if the following properties are satisfied

    \begin{enumerate}[i)]
        \item Every face of a simplex in $\mcT$ is also
        in $\mcT$. 

        \item The non-empty intersection of any two simplices $\tau_1,\tau_2 \in \mcT$
        is a face of both $\tau_1$ and $\tau_2$. 
    
        \item The highest dimension of any simplex in $\mcT$ is $k \leq d$.

\end{enumerate}

\end{defi}

This definition allows us to construct many general domains in
$\R^{d}$ formed by $k$-simplices. For example, in $\R^{3}$ we can
build a triangulation of a surface with $2$-simplices.

\begin{defi}
  Let $\bx_i$ be the barycenter of any simplex $\tau_i \in
  \mcT$, and $\bbS := \{\bx_1,\dots,\bx_N\}$.
\end{defi}

\begin{asum} We make the following assumptions for $\mcT$ on the domain $U \subset \R^{d}$:

    \begin{enumerate}[i)]

        \item $\mcT$ contains $N$ simplices $\tau_i$, for $i = 1,\dots,N$, of the same order.
        
        \item $U = \cup_{\tau_i \in \mcT} \tau_i $.

        \item 
        For any $i = 1,\dots,N$ and for any simplex $\tau_i \in \mcT$
        let $\chi_{i} = c_i 1_{\tau_i}$. The coefficients $c_i$ for $i
        = 1, \dots, N$ are chosen such that $\chi_1, \dots, \chi_N$
        form an orthonormal set in $L^{2}(U)$.
    
        \item Let $\mcE:=\{\chi_{1},\dots,\chi_{N}\}$ and $V_n =
          \mcP(\mcE) := span \{\chi_1, \dots, \chi_N \}$. Assume that
          Karhunen-Lo\`{e}ve eigenfunctions $\phi_i \in \mcP(\mcE)$
          for all $i = 1,\dots,M$ where $N > M$.
    
    \end{enumerate}  
\label{mls:assum3}
\end{asum}

 With the goal of constructing a multi-level basis, the domain $U$ is
 in general embedded in a kd-tree type decomposition.  This will allow
 the MB construction algorithm to efficiently access the simplices of
 $\mcT$ by searching a binary tree. The approach is described in
 \cite{Castrillon2020} in the context of discrete vector spaces.  For
 very high dimensional domains alternative choices also include Random
 Projection (RP) trees \cite{Dasgupta2008}.

Suppose that all the barycenters $\bx \in \bbS$ are embedded in the
cell $B^{0}_{0} \subset \R^d$, which corresponds to the top of the
binary tree.  Without loss of generality it can be assumed that
$B^{0}_{0} = [0,1]^{d}$ and $\bbS \subset B^{0}_0$.  Each cell
$B^{l}_{k}$ at the level $l$ of the tree and index $k$ contains a
subset of the barycenters in $\bbS$.  There is a subdivision of any cell
$B^{l}_{k}$ into two cells $B^{l-1}_{{\tt left}}$ and $B^{l-1}_{\tt right}$
according to the following rule (see Algorithm
\ref{RPMLB:algorithm2-kd}):

\begin{enumerate}[1)]

\item For each coordinate $1 \leq j \leq d$, project every barycenter $\bx_i \in 
B^{l}_{k}$ onto  the unit vector along coordinate $k$ and compute the sample variance of these projection coefficients.

\item Choose the unit coordinate vector $v$ in the direction $1\leq j \leq d$ with the maximal 
sample variance for the above projection coefficients.

\item Compute the median of the projections along $v$ and split the cell in two at this coordinate position.
  
\end{enumerate}

\begin{algorithm}[h]
  \KwIn{ $\tilde \bbS$}
  \KwOut{Rule, threshold, $v$}
\Begin{
    choose a coordinate direction that has maximal variance of the projection
    of the points in $\tilde \bbS$. \\
Rule(x) := $x \cdot v  \leq$ threshold = median
}

\caption{ChooseRule($\tilde \bbS$) function for kD-tree where $\tilde
  \bbS \subset \bbS$.}
\label{RPMLB:algorithm2-kd}
\end{algorithm}

The initial cell $B^{0}_{0}$ is subdivided in this manner until a
maximum number of $n_0$ barycenters are located at each of the leaf
cells.  It is also assumed that $n_0 > M$, with $M$ the above number of
truncated KL coefficients.
Let $\mcB^{l}$ be the collection of all the cells $B^{l}_{k}$
at level $l$.  Algorithm \ref{RPMLB:algorithm1} below describes in more
detail the construction of the binary tree. In Figure \ref{MLRLE:fig1}
an example of the kd-Tree partitioning of a triangulation $\mcT$ is
shown along with the associated binary tree.

\setlength{\tabcolsep}{36pt}
\tikzset{edge from parent/.style={draw,edge from parent path={(\tikzparentnode.south)-- +(0,-8pt)-| (\tikzchildnode)}}}

\begin{figure*}[htb]
\begin{center}
\begin{tabular}{c c }
\begin{tikzpicture}[scale=.65]

\begin{scope}
[place/.style={circle,draw=rufous,fill=rufous,thick, inner sep=0pt,minimum size=1mm}]
\begin{scope}[xshift = 2cm, yshift = 3.5cm]
\draw (0,0) to (0.5,0.8660);
\draw (0,0) to (-0.5,0.8660);
\draw (0,0) to (-1,0);
\draw (0,0) to (-0.5,-0.8660);
\draw (0,0) to (0.5,-0.8660);
\draw (0.5,0.8660) to (-0.5,0.8660); 
\draw (-0.5,0.8660) to (-1,0);
\draw (-1,0) to (-0.5,-0.8660);
\draw (-0.5,-0.8660) to (0.5,-0.8660);

\node at (0,0.5774) [place] {};
\node at (-0.50,0.2887) [place] {};
\node at (-0.50,-0.2887) [place] {};
\node at (0,-0.5774) [place] {};
\end{scope}

\begin{scope}[xshift = 2cm, yshift = 1.25cm, rotate = 90]
\draw (0,0) to (0.5,0.8660);
\draw (0,0) to (-0.5,0.8660);
\draw (0,0) to (-1,0);
\draw (0,0) to (-0.5,-0.8660);
\draw (0,0) to (0.5,-0.8660);
\draw (0.5,0.8660) to (-0.5,0.8660); 
\draw (-0.5,0.8660) to (-1,0);
\draw (-1,0) to (-0.5,-0.8660);
\draw (-0.5,-0.8660) to (0.5,-0.8660);

\node at (0,0.5774) [place] {};
\node at (-0.50,0.2887) [place] {};
\node at (-0.50,-0.2887) [place] {};
\node at (0,-0.5774) [place] {};
\end{scope}

\begin{scope}[xshift = 2cm, yshift = 7cm]
\draw (0,0) to (0.5,0.8660);
\draw (0,0) to (-0.5,0.8660);
\draw (0,0) to (-1,0);
\draw (0,0) to (-0.5,-0.8660);
\draw (0,0) to (0.5,-0.8660);
\draw (0.5,0.8660) to (-0.5,0.8660); 
\draw (-0.5,0.8660) to (-1,0);
\draw (-1,0) to (-0.5,-0.8660);
\draw (-0.5,-0.8660) to (0.5,-0.8660);

\node at (0,0.5774) [place] {};
\node at (-0.50,0.2887) [place] {};
\node at (-0.50,-0.2887) [place] {};
\node at (0,-0.5774) [place] {};
\end{scope}

\begin{scope}[xshift = 3cm, yshift = 7cm]
\draw (0,0) to (0.5,0.8660);
\draw (0,0) to (-0.5,0.8660);
\draw (0,0) to (-1,0);
\draw (0,0) to (-0.5,-0.8660);
\draw (0,0) to (0.5,-0.8660);
\draw (0.5,0.8660) to (-0.5,0.8660); 
\draw (-0.5,0.8660) to (-1,0);
\draw (-1,0) to (-0.5,-0.8660);
\draw (-0.5,-0.8660) to (0.5,-0.8660);

\node at (0,0.5774) [place] {};
\node at (-0.50,0.2887) [place] {};
\node at (-0.50,-0.2887) [place] {};
\node at (0,-0.5774) [place] {};
\end{scope}

\begin{scope}[xshift = 4cm, yshift = 7cm]
\draw (0,0) to (0.5,0.8660);
\draw (0,0) to (-0.5,0.8660);
\draw (0,0) to (-1,0);
\draw (0,0) to (-0.5,-0.8660);
\draw (0,0) to (0.5,-0.8660);
\draw (0.5,0.8660) to (-0.5,0.8660); 
\draw (-0.5,0.8660) to (-1,0);
\draw (-1,0) to (-0.5,-0.8660);
\draw (-0.5,-0.8660) to (0.5,-0.8660);

\node at (0,0.5774) [place] {};
\node at (-0.50,0.2887) [place] {};
\node at (-0.50,-0.2887) [place] {};
\node at (0,-0.5774) [place] {};
\end{scope}

\begin{scope}[xshift = 6.95cm, yshift = 5.01cm]
\draw (0,0) to (0.5,0.8660);
\draw (0,0) to (-0.5,0.8660);
\draw (0,0) to (-1,0);
\draw (0,0) to (-0.5,-0.8660);
\draw (0,0) to (0.5,-0.8660);
\draw (0.5,0.8660) to (-0.5,0.8660); 
\draw (-0.5,0.8660) to (-1,0);
\draw (-1,0) to (-0.5,-0.8660);
\draw (-0.5,-0.8660) to (0.5,-0.8660);

\node at (0,0.5774) [place] {};
\node at (-0.50,0.2887) [place] {};
\node at (-0.50,-0.2887) [place] {};
\node at (0,-0.5774) [place] {};
\end{scope}

\begin{scope}[xshift = 6.95cm, yshift = 3.28cm, rotate = 180]
\draw (0,0) to (0.5,0.8660);
\draw (0,0) to (-0.5,0.8660);
\draw (0,0) to (-1,0);
\draw (0,0) to (-0.5,-0.8660);
\draw (0,0) to (0.5,-0.8660);
\draw (0.5,0.8660) to (-0.5,0.8660); 
\draw (-0.5,0.8660) to (-1,0);
\draw (-1,0) to (-0.5,-0.8660);
\draw (-0.5,-0.8660) to (0.5,-0.8660);

\node at (0,0.5774) [place] {};
\node at (-0.50,0.2887) [place] {};
\node at (-0.50,-0.2887) [place] {};
\node at (0,-0.5774) [place] {};
\end{scope}

\begin{scope}[xshift = 5.95cm, yshift = 3.28cm, rotate = 180]
\draw (0,0) to (0.5,0.8660);
\draw (0,0) to (-0.5,0.8660);
\draw (0,0) to (-1,0);
\draw (0,0) to (-0.5,-0.8660);
\draw (0,0) to (0.5,-0.8660);
\draw (0.5,0.8660) to (-0.5,0.8660); 
\draw (-0.5,0.8660) to (-1,0);
\draw (-1,0) to (-0.5,-0.8660);
\draw (-0.5,-0.8660) to (0.5,-0.8660);

\node at (0,0.5774) [place] {};
\node at (-0.50,0.2887) [place] {};
\node at (-0.50,-0.2887) [place] {};
\node at (0,-0.5774) [place] {};
\end{scope}

    \node at (7.5,8.5) [] {$B^{0}_0$};
    \node at (1,5.5) [] {$B^{3}_{7}$};
    \node at (2.75,5.5) [] {$B^{3}_{8}$};
    
 \draw[step=8,gray,very thin] (0, 0) grid (8, 8);
    \draw (3.25,0) to (3.25,8);
    
    \draw (0,5) to (3.25,5);
    \draw (2.2,5) to (2.2,8);
    \draw (0,2) to (3.25,2);

    \draw (0,5) to (3.25,5);

    \draw (3.25,4.15) to (8,4.15);
    \draw (5,4.15) to (5,8);
    \draw (6.75,0) to (6.75,4.15);
    
  \end{scope}

\end{tikzpicture} 
&

\raisebox{0.5cm}{
\begin{tikzpicture}[scale=0.85]
\begin{scope}[xshift=5cm, yshift=5cm,
place/.style={circle,draw=blue!50,fill=blue!20,thick,
      inner sep=0pt,minimum size=1.5mm},
placer/.style={circle,draw=darkgreen!50,
  preaction={fill=olivine,opacity = 0.5}, thick,inner
  sep=0pt,minimum size=1.5mm}, ]


  
\Tree [.\node[placer]{$B^{0}_{0}$}; 
             [.\node[placer]{$B^{1}_{1}$};
                    [.\node[placer]{$B^{2}_{3}$}; 
                           [.\node[placer]{$B^{3}_{7}$};]
                           [.\node[placer]{$B^{3}_{8}$};] 
                    ]       
                    [.\node[placer]{$B^{2}_{4}$}; 
                           [.\node[placer]{$B^{3}_{9}$};] 
                           [.\node[placer]{$B^{3}_{10}$};] 
                    ] 
             ]                                        
             [.\node[placer]{$B^{1}_{2}$};
                    [      [.\node[placer]{$B^{2}_{5}$};
                                  [.\node[placer]{$B^{3}_{11}$};] 
                                  [.\node[placer]{$B^{3}_{12}$};] 
                           ]
                           [.\node[placer]{$B^{2}_{6}$}; 
                                  [.\node[placer]{$B^{3}_{13}$};] 
                                  [.\node[placer]{$B^{3}_{14}$};] 
                           ] 
                                          ]]
]


\end{scope}
\end{tikzpicture}
}
\end{tabular}
\end{center}
\caption{Multilevel kd-tree domain decomposition of a 
triangulation $\mcT$ with respective binary tree. Assume that the
tree has $l = 0,\dots,n$ levels.}
\label{MLRLE:fig1}
\end{figure*}
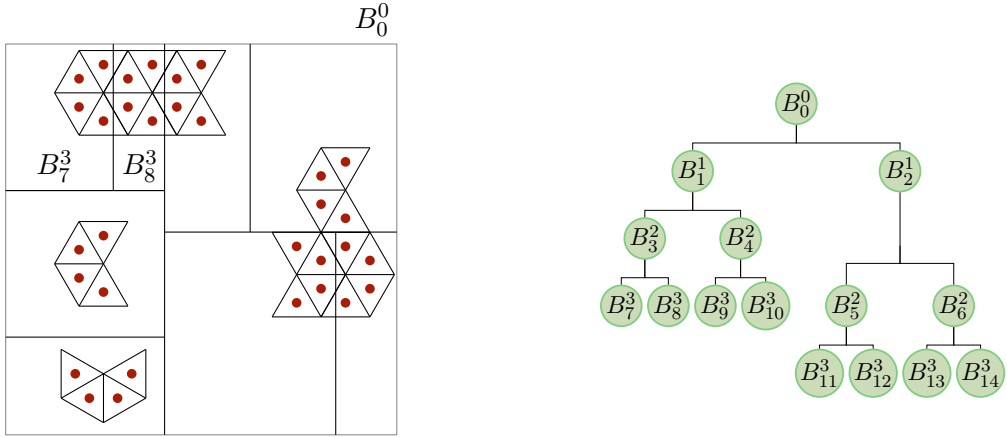

\setlength{\tabcolsep}{36pt}

\begin{algorithm}[h]
  \KwIn{ $\bbS$, node, $n_0$} \KwOut{Tree}

\Begin{

\eIf {Tree = root}{node $\leftarrow$ 0, currentdepth $\leftarrow$ 0
Tree $\leftarrow$ MakeTree($\bbS$, node,
currentdepth + 1, $n_0$)
}
{

Tree.node = node

Tree.currentdepth = currentdepth - 1

node $\leftarrow$ node + 1

\If {$|\tilde \bbS| < n_0$}{return (Leaf)}

(Rule, threshold, $v$) $\leftarrow$ ChooseRule($\tilde \bbS$)

(Tree.LeftTree, node) 
$\leftarrow$ MakeTree($\bx \in \tilde \bbS$: Rule($\bx$) = True, node,
currentdepth + 1, $n_0$)

(Tree.RightTree, node)
$\leftarrow$ MakeTree($\bx \in \tilde \bbS$: Rule($\bx$) = false,
node, currentdepth + 1, $n_0$)

Tree.threshold = threshold\\
Tree.$v$ = $v$
}
}
\caption{MakeTree($\bbS$) function}
\label{RPMLB:algorithm1}
\end{algorithm}
\corb{
\begin{rem}
Note that if the number of barycenters is even then the
tree will end evenly at some level $n$. However, if
the number is odd and $n_0 = 1$ then one branch can end at level $n$ and the other at level $n - 1$.
\end{rem}
}
From Assumption \ref{mls:assum3} the set of barycenter locations $\{
\bx_{1},\dots,\bx_{N}\}$ have a one to one correspondence with $\mcE$,
i.e. $\bx_{t} \longleftrightarrow \chi_{t}$ for all $t = 1, \dots,
N$. The multilevel basis construction algorithm adapted to the
Karhunen Lo\`{e}ve expansion $v_{M}(\bx,\omega)$ is described as
follows:

\begin{enumerate}[(I)]
\item Start at the finest level of the tree,
  i.e. $q = n$.
\item For each leaf cell $B^{q}_{k} \in \mcB^{q}$ assume without loss
  of generality that there are $s$ barycenters $\bbS^{q}:=\{ \bx_1,
  \dots, \bx_s \}$ with associated functions $\mcE_k^{q} := \{ \chi_1,
  \dots, \chi_s \}$. Let $\mcQ^{q}_{k}(\mcE_k^{q})$ be the span of the
  functions in $\mcE_k^{q}$.
\begin{enumerate}[i)]

\item Let $\bphi^{q,k} _{j} := \sum_{\chi_i \in \mcE^q_k} 
c^{q,k} _{i,j}
  \chi_i, \hspace{2mm} j=1, \dots, a_{q,k};
\hspace{2mm} \bpsi^{q,k}_{j} := \sum_{\chi_i \in \mcE^q_k} 
d^{q,k}_{i,j}
\chi_i, \hspace{2mm} j=a_{q,k}+1, \dots, s$, where $c^{q,k}_{i,j}$,
$d^{q,k}_{i,j} \in \mathbb{R}$, and are undetermined at the moment, 
and for some yet undetermined $a_{q,k} \in
\mathbb{N}_0$. The objective of the above linear combinations is to
construct new functions $\bpsi^{q,k}_{j}$ orthogonal to $V_0$,
i.e., such that for all $\phi_i \in V_0$, $i = 1, \dots,M$, we have that
\begin{equation}
\int_{U}
\phi_i(\bx)  \bpsi^{q,k}_{j}(\bx) \, \mbox{d} \bx 
= 0.
\label{hbconstruction:eqn1}
\end{equation}


\item From the eigenfunctions $\phi_1, \dots, \phi_M$ of the KL expansion and
  $\mcE^{q}_{k}$ we can form the \corb{matrices}



\corb{ 
\[
\bM^{q,k} := 
\begin{bmatrix}
 \langle \phi_{1}(\bx),  \chi_{1}(\bx)\rangle  & 
\dots & 
 \langle \phi_{1}(\bx),  \chi_{s}(\bx)\rangle   \\
 \langle \phi_{2}(\bx),  \chi_{1}(\bx)\rangle  & 
\dots & 
 \langle \phi_{2}(\bx),  \chi_{s}(\bx)\rangle   \\
\vdots & \ddots & \vdots \\
 \langle \phi_{M}(\bx),  \chi_{1}(\bx)\rangle  & 
\dots & 
\langle \phi_{M}(\bx) , \chi_{s}(\bx)\rangle   \\
\end{bmatrix}
\]
and
\[
\bN^{q,k} := 
\left[
\begin{array}{c | c}
\begin{matrix}
 \langle \phi_{1},  \bphi^{q,k}_{1} \rangle  & 
\dots & 
 \langle \phi_{1},  \bphi^{q,k}_{a_{q,k}} \rangle   \\
\vdots & \ddots & \vdots \\
 \langle \phi_{M},  \bphi^{q,k}_{1} \rangle  & 
\dots & 
\langle \phi_{M} , \bphi^{q,k}_{a_{q,k}} \rangle   \\
\end{matrix}
&
\begin{matrix}
 \langle \phi_{1},  \bpsi^{q,k}_{a_{q,k} + 1} \rangle  & 
\dots & 
 \langle \phi_{1},  \bpsi^{q,k}_{s} \rangle   \\
\vdots & \ddots & \vdots \\
 \langle \phi_{M},  \bpsi^{q,k}_{a_{q,k} + 1} \rangle  & 
\dots & 
 \langle \phi_{M} , \bpsi^{q,k}_{s} \rangle   \\
\end{matrix}
\end{array}
\right]
\]
where $\langle\cdot,\cdot\rangle$ is the $L^{2}(U)$ inner product.}

\item Apply the Singular Value Decomposition (SVD) to $\bM^{q,k}$
  i.e.  $\bM^{q,k} = \bU \bD \bV$ where $\bU \in \R^{ M \times M}$,
  $\bD \in \R^{M \times s}$, and $\bV \in \R^{s \times s}$. Assume that $a_{q,k}$ is the rank of
  the matrix $\bM^{q,k}$.

\item Consider the following choice for the coefficients $c^{q,k}_{i,j}$ and
$d^{q,k}_{i,j}$ from the right SVD matrix:
\[
  \left[ \begin{array}{ccc|ccc}
      c^{q,k}_{0,1} & \dots &c^{q,k}_{a_{q,k},1} & d^{q,k}_{a_{q,k}+1,1} & \dots &d^{q,k}_{s,1} \\
      c^{q,k}_{0,2} & \dots &c^{q,k}_{a_{q,k},2} & d^{q,k}_{a_{q,k}+1,2} & \dots &d^{q,k}_{s,2} \\
      \vdots & \vdots & \vdots & \vdots & \vdots & \vdots   \\
      c^{q,k}_{0,s} & \dots &c^{q,k}_{a_{q,k},s} & d^{q,k}_{a_{q,k}+1,s} & \dots &d^{q,k}_{s,s}
    \end{array}
\right]
 := \bV\T .
\]
Since the vectors (column and row) in $\bV$ are orthonormal in an $l_2(\R^s)$ 
sense and
the functions in $\mcE^{q}_{k}$ are orthonormal in $L_2(U)$ then
$\bphi^{q,k}_{1}, \dots, \bphi^{q,k}_{{a}_{q,k}}$, and
$\bpsi^{q,k}_{a_{q,k}+1}, \dots,$ $\bpsi^{q,k}_{s}$ form an
orthonormal basis of $\mcQ^{q}_{k}(\mcE_k^{q})$.  \corb{As in 
\cite{Tausch2001, Castrillon2016a} it can be shown 
that this choice leads to $\bpsi^{q,k}_{a_{q,k}+1}, \dots, \bpsi^{q,k}_s$ satisfying equation \eqref{hbconstruction:eqn1}.
From the SVD of $\bM^{q,k}$ and this choice of coefficients
\[
\bN^{q,k} = \bM^{q,k} \bV = \bU \bD.
\]
Now, decompose $\bD$ as  $\bD = [ \bSigma \, |\, \0 ]$, where
$\bSigma \in \R^{M \times a_{q,k}}$ is a diagonal matrix with 
the non-zero singular values of $\bM^{q,k}$ and  the zero matrix $\0 \in  \R^{M \times (s - a_{q,k}) }$. Thus $\bU \bD = [\bU \bSigma\,|\,\0]$
and $\bN^{q,k} = \bM^{q,k} \bV = [\bU \bSigma\,|\,\0]$. It follows
that SVD columns $a_{q,k}+1,\dots,s$ of $\bV$ form an orthonormal 
basis of the nullspace of $\bM^{q,k}$ and therefore 
$\bpsi^{q,k}_{a_{q,k}+1}, \dots, \bpsi^{q,k}_s$ 
satisfy equation \eqref{hbconstruction:eqn1}, and are
orthogonal to $V_0$ and compactly supported in the cell $B^{q}_{k}$.
}

\item 
Denote by $D_k^{q} :=\{\bpsi^{q,k}_{a_{q,k}+1}, \dots, 
\bpsi^{q,k}_{s} \}$ and  $C_k^{q}:=\{ \bphi^{q,k}_{1}, 
\dots, \bphi^{q,k}_{a_{q,k}}\}$.

  \begin{rem}
    Notice that the functions
  $\bphi^{q,k}_{1}, \dots,$ $\bphi^{q,k}_{a_{q,k}}$
  are {\it not} in general
  orthogonal to $V_0$. 
  \end{rem}

\end{enumerate}

\item Let $\mcD^{q} := \cup_{B^{q}_{k} \in \mcB^{q} 
}D^{q}_k$ and $\mcC^{q} :=
  \cup_{B^{q}_{k} \in \mcB^{q}} 
  C^{q}_k$.  It is not hard to see that they
  form an orthonormal set in $L^{2}(U)$.  Denote by $W_{q-1}$ the
  span of all the functions in $\mcD^{q}$ and similarly $V_{q-1}$ with
  respect to $\mcC^{q}$.

\item The next step is to go to level $q - 1$. For any two sibling
  cells denoted as $B^{q}_{\tt{left}}$ and $B^{q}_{\tt{right}}$ at level $q$ denote 
  $\mcE^{q-1}_{\tilde k}$, for some index $\tilde k$, as the
  collection of functions $\bphi^{q,\tt{left}}_{1}, \dots,$
  $\bphi^{q,\tt{left}}_{a_{q,\tt{left}}}$ and $\bphi^{q,\tt{right}}_{1}, \dots,$
  $\bphi^{q,\tt{right}}_{a_{q,\tt{right}}}$

\item Recursively, let $q: = q - 1$. If $B^{q}_{k} \in \mcB^{q}$ is a leaf cell (which may occur as not all branches of the tree 
\corb{necessarily} have the same numbers of levels)
  then go to (II). However, if $B^{q}_{k} \in \mcB^{q}$ is not a leaf
  cell, then go to (II) but replace the collection of leaf cell
  functions with $\mcE^{q}_{k}:=\{ \bphi^{q+1,\tt{left}}_{1}, \dots,$
  $\bphi^{q+1,\tt{left}}_{a_{q+1,\tt{left}}}$, $\bphi^{q+1,\tt{right}}_{1}, \dots,$
  $\bphi^{q+1,\tt{right}}_{a_{q+1,\tt{right}}}$ $\}$.

  \item When $q = -1$ is reached then incrementation stops.
  

\end{enumerate}

When the algorithm terminates a series of orthogonal subspaces $W_{0},\dots,
W_{n}$ and corresponding basis functions $\mcD^{0}, \dots
\mcD^{t}$ are obtained. Furthermore, it can be shown that $V_0 =
\mbox{span} \{ \phi_1, \dots \phi_M\}$ is also the span of
$\{\bphi^{0,0}_{1}, \dots,$ $\bphi^{0,0}_{a_{0,0}}\}$ and $a_{0,0} =
M$.

\begin{rem} Following the arguments in \cite{Castrillon2013,Castrillon2020}
  it can be shown that

\begin{itemize}

\item $V_n = \mcP(\mcE) = V_0 \oplus W_{0} \oplus W_{1} \oplus \dots
  \oplus W_{n-1}$

\item $\mcC^{0}$, $\mcD^{0}$, $\mcD^{1}$, $\dots$  $\mcD^{n-1}$ form an
orthonormal basis for $V_0 \oplus
W_{0}
\oplus W_{1}
\oplus \dots \oplus W_{n-1}$

\item At most $\mcO(Nn)$ computational steps are needed to construct
  the multilevel basis of $V_n$.

\item Let $\gamma \in \mcP(\mcE)$ and denote $c_{1}, \dots c_{N}$ the
  orthogonal projection coefficients on $\mcP(\mcE)$ where
\[
c_{i} = \int_U \gamma(\bx) \chi_{i}(\bx) \,\mbox{d}\bx
\]
and $i = 1,\dots,N$.  It can be shown that the multilevel 
projection coefficients $d^{l}_{k}$ for $l = 0, \dots n-1$ can be 
computed in at most $\mcO(Nn)$ computational steps and memory
from $c_1,\dots,c_N$.
\end{itemize}

\end{rem}

\begin{rem}
For many practical situations only spatial samples of $u(\bx,\omega)$
are available. For such cases the alternative choice for $\mcE$ is a
set of unit vectors. A similar construction to the multilevel basis
can be done in a vector sense (see
\cite{Castrillon2020,Castrillon2013} for details). This is equivalent
to the continuous multilevel basis construction, up to a re-scaling of
the domain, by assuming that each simplex in $\mcT$ has the same unit
volume measure and $u(\bx,\omega)$ is constant on each simplex. The
multilevel coefficients from Examples \ref{mls:example1} and
\ref{mls:example2} where obtained with the discrete version of the
multilevel basis using 500 equally spaced samples.
\end{rem}

\begin{rem}
  The algorithm is efficiently implemented in MATLAB \cite{Matlab2018}
  and can handle highly complex geometries.  The code will be made
  available to the general public.
\end{rem}

\section{Spherical example}
\label{spherical}

We will now demonstrate the application of the multilevel orthogonal
eigenspace for the detection of signals on Spherical Fractional
Brownian Motion (SFBM) defined on the unit sphere $\bbS_2$ \cite{Istas2006}.
This is a more complex scenario that shows the flexibility of this
approach.

Suppose $\{P^m_l(x)\}_{l\geq 0}$ is the set of associated Legendre
polynomials for $m \geq 0$. If $m$ is negative then the associated
Legendre polynomials are given by

\[
P^{-m}_l(x) = (-1)^{m} \frac{(l-m)!}{(l+m)!} P^m_l(x).
\]

The coordinates of the unit sphere $\bbS_2$ are given by the
colatitude $\theta \in [0, \pi)$ and longitude $\varphi \in [0,2 \pi)$.
    The spherical harmonics $Y^m_l(\theta,\varphi)$ on $\bbS_2$ are
    defined by
 \[
  Y^m_l(\theta,\varphi)= \sqrt{ \frac{2l + 1}{4 \pi }
    \frac{(l-m)!}{(l+m)!}  } P^{m}_{l}(cos\, \theta)e^{im\varphi},
  \]
where $m = -l,\dots,l$.  In \cite{Istas2006} the author demonstrates
that the Karhunen-Lo\`{e}ve expansion of the SFBM is given by
\[
v(\theta,\varphi,\omega) = \sum_{l \geq 0} \sum_{m = -l}^{l} \sqrt{
  -\pi d_{l}} \varepsilon^{m}_{l}(\omega)(Y^m_l(\theta,\varphi) -
  Y^m_l(0,0))
  \]
  in the $L^{2}$ sense, where $\varepsilon^{m}_{l} \sim N(0,1)$
  i.i.d. and
  \[
  d_l := \int_{-1}^{1} \text {arccos}(x) P_l(x)\,\mbox{d}x.
  \]

  \begin{rem} Notice the spherical
    harmonics $Y^m_l(\theta,\varphi):[0,\pi] \times [0, 2\pi]
    \rightarrow \bbC$ are complex-valued, which however does not restrict this analysis.  Although the theoretical discussion of the
    multilevel orthogonal eigenspaces is given for real
    Hilbert spaces, the method can be readily extended to the complex
    case. In particular, the algorithm implementation can also handle
    this case.
    \end{rem}

  In Figure \ref{mls:fig3} three realizations from the KL expansion of
  the absolute value of the SFBM are shown. For this example $l = 10$,
  which is sufficient to capture much of the stochastic movement since
  it is shown in \cite{Istas2006} that $d_l$ decays as $l^{-2}$.
  However, since $d_l = 0$ whenever $l = 3, 5, 7, 9$ then the
  truncated KL expansion is reduced to $M = 56$ eigenfunctions.

\begin{figure}[htb]
    \centering
    \begin{tikzpicture}[scale = 0.90]
    \node at (-6,0)
    {
        \includegraphics[scale = 0.4,
        trim = 6cm 6.5cm 4cm 7cm,
        clip=true]{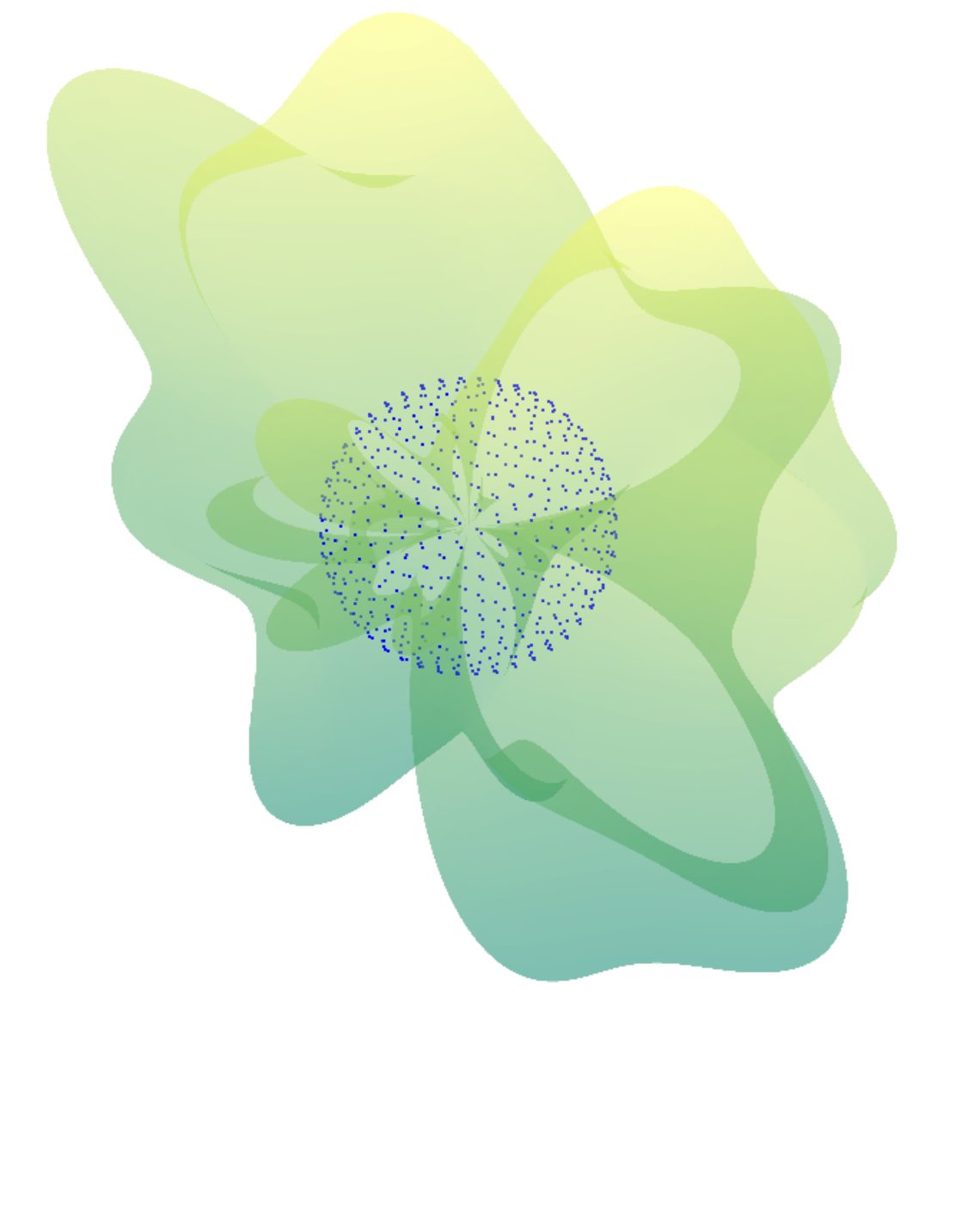}
    };
    
    \node at (0,0)
    {
        \includegraphics[scale = 0.4,
        trim = 5cm 6.5cm 1cm 7cm,
        clip=true]{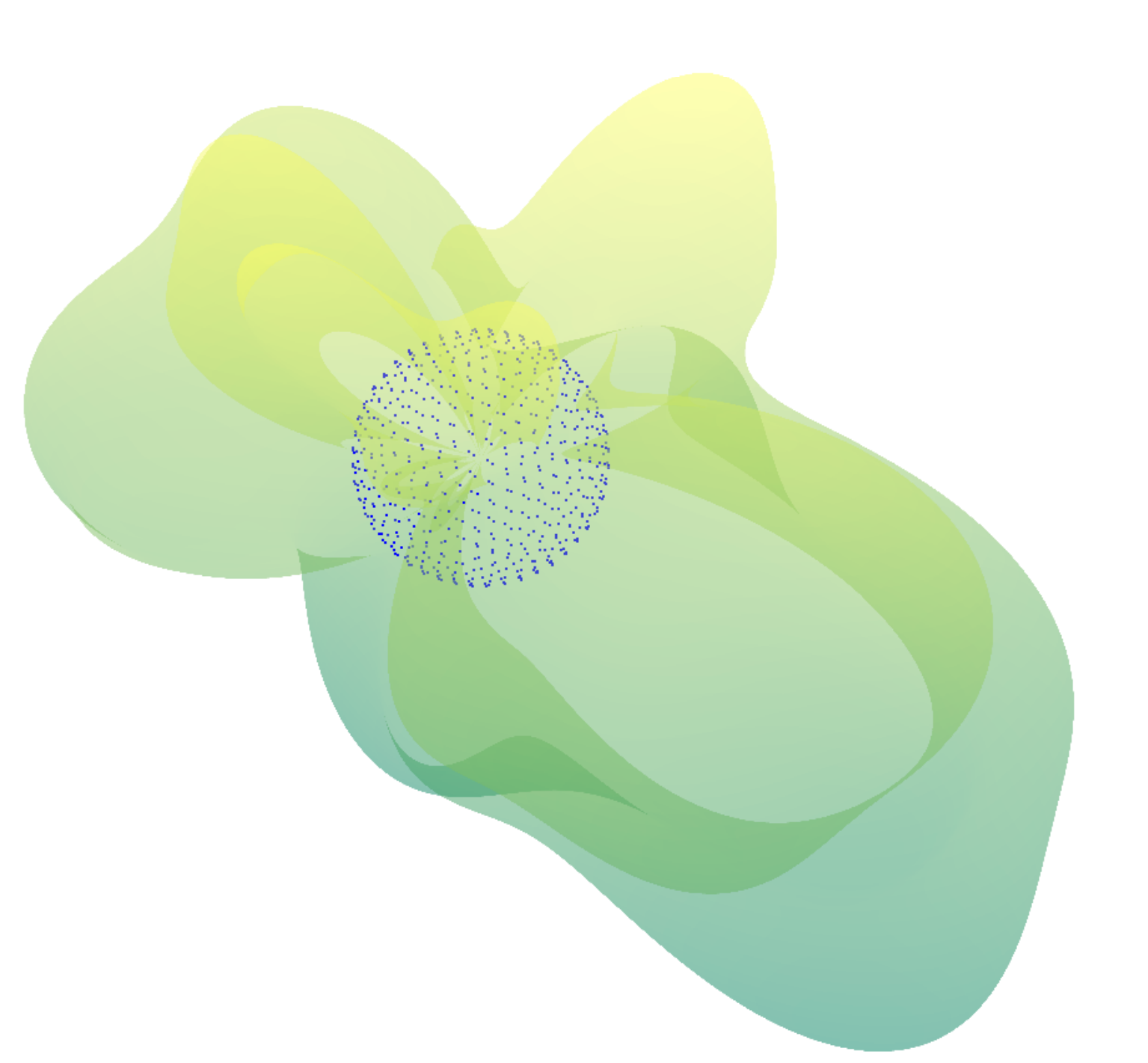}
    };
    
    \node at (6,0)
    {
        \includegraphics[scale = 0.4,
        trim = 3cm 6.5cm 5cm 7cm,
        clip=true]{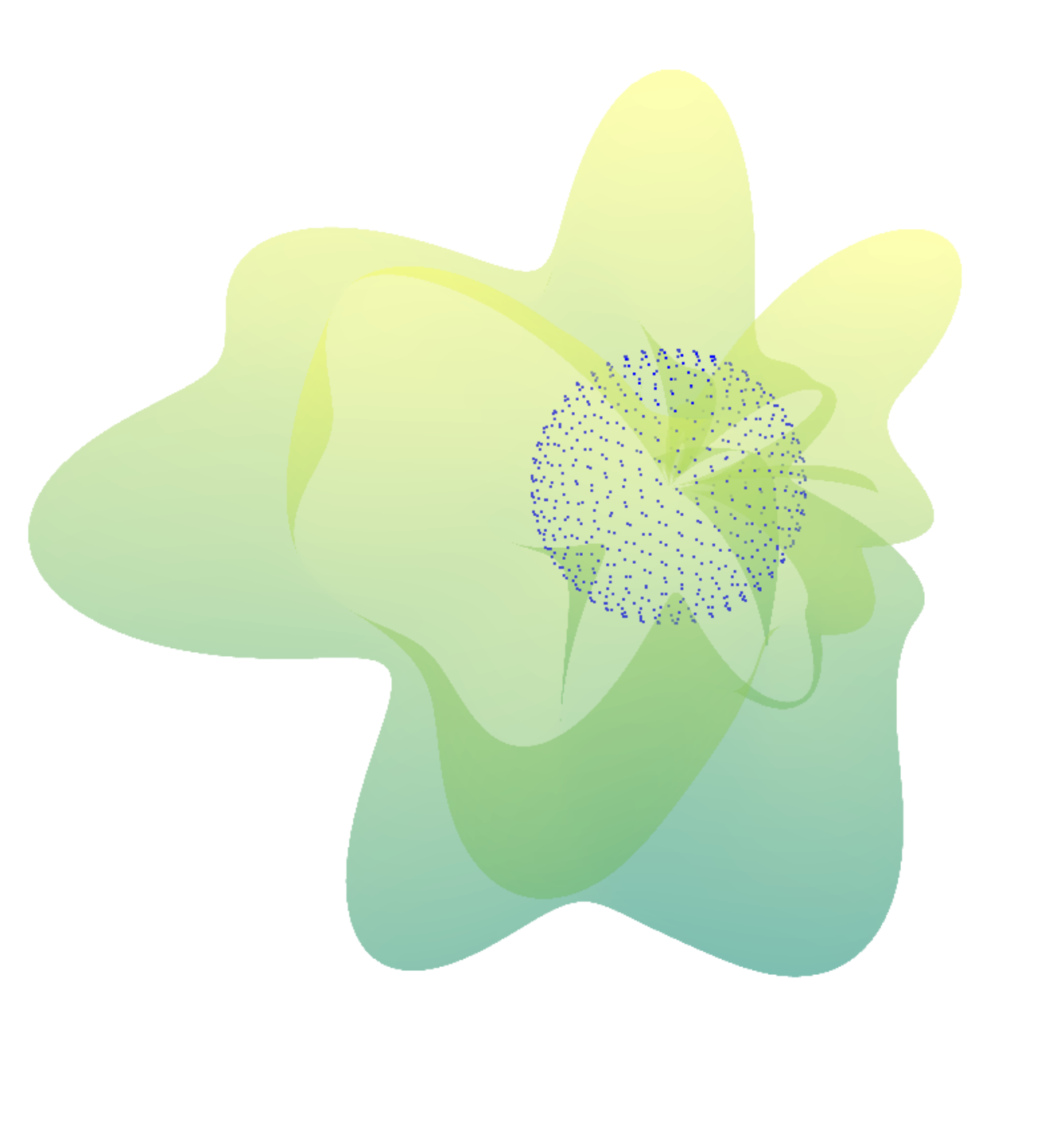}
    };

     \node at (-6,-4) { $v(\bx,\omega_1)$};
     \node at ( 0,-4) { $v(\bx,\omega_2)$};
     \node at ( 7,-4) { $v(\bx,\omega_3)$};
         
    \end{tikzpicture}
    \caption{Realizations of the Spherical Fraction Brownian Motion (SFBM) from
      $l = 10$  ($M = 56$ eigenfunctions) Karhunen Lo\`{e}ve expansion. The blue dots
  correspond to the sampling of the unit sphere.}
    \label{mls:fig4}
\end{figure}
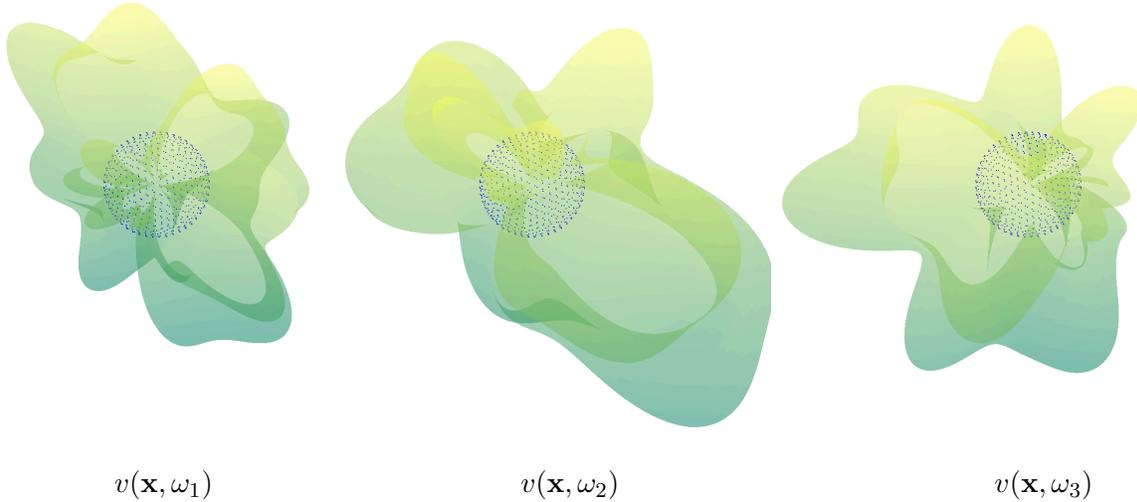

  Suppose we apply a perturbation to $v(\theta,\varphi,\omega)$ of the
  form
  \[
  w(\theta,\varphi) = c \exp \left( - \frac{(\theta - \pi/2)^2 
    + (\varphi - \pi/2)^2  }{\sigma^{2}} \right)
  \]
  where $c = 0.5$ and $\sigma = 0.1$, i.e. $u(\theta,\varphi,\omega) =
  v + w$.  The goal is to detect $w$ with the ML orthogonal
  eigenspace. For this case we sample the sphere with 10,242 almost
  equally spaced barycenters \cite{Vanlaven2010} and construct the ML
  basis.

  After applying the ML to the spherical signal $u$, we analyze the
  projection coefficients corresponding to the spaces $W_{n-1}$,
  $W_{n-2}$ and $W_{n-3}$ where $n = 8$.  In Figure \ref{mls:fig5} the
  locations of the basis functions corresponding to the ML projection
  coefficients are shown for any coefficient with absolute value
  greater than $10^{-4}$. Over imposed with on each sphere is the
  Gaussian perturbation $w(\phi,\vartheta)$. Notice that these
  coefficients essential identify the location of the Gaussian from
  the original signal $u$ at different levels of resolution.
  Furthermore, we can use these coefficients to estimate the size of
  the perturbation from Theorem \ref{mls:theo3}.

\begin{figure}[htb]
    \centering
    \begin{tikzpicture}[scale = 0.90]
    \node at (-6,0)
    {
        \includegraphics[scale = 0.4,
        trim = 5cm 6.5cm 1cm 7cm,
        clip=true]{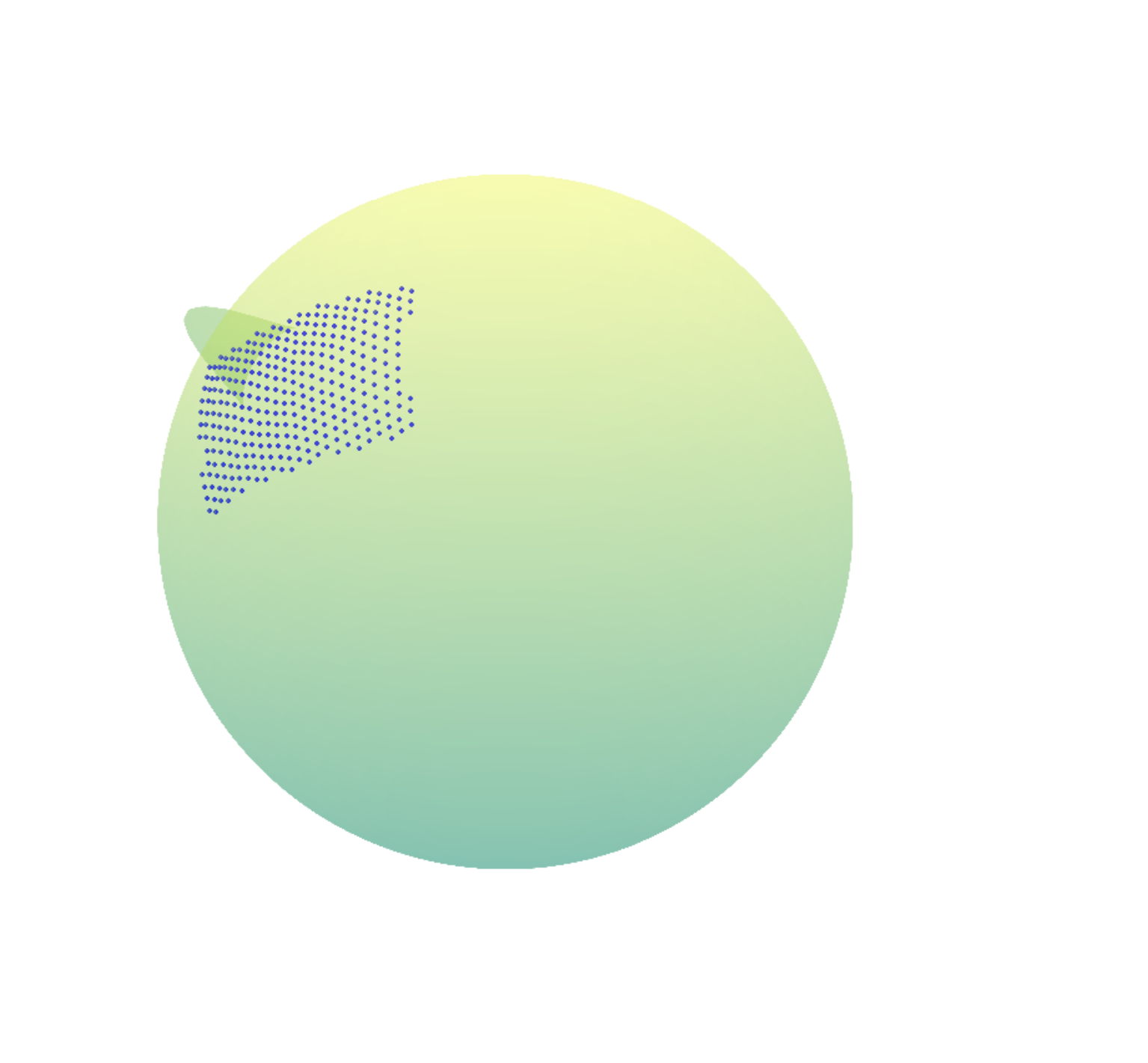}
    };
    
    \node at (0,0)
    {
        \includegraphics[scale = 0.4,
        trim = 5cm 6.5cm 1cm 7cm,
        clip=true]{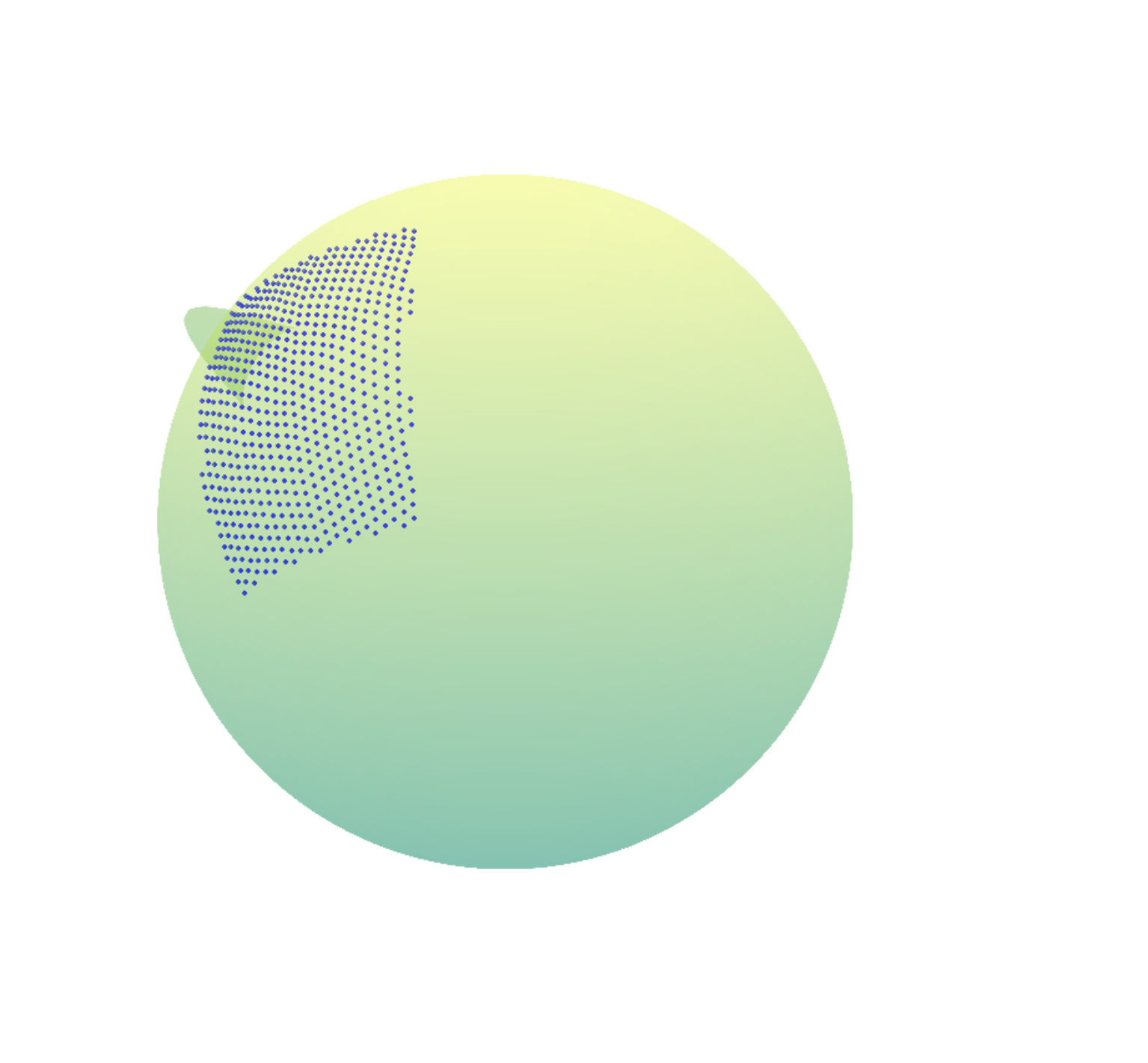}
    };
    
    \node at (6,0)
    {
        \includegraphics[scale = 0.4,
        trim = 5cm 6cm 1cm 6cm,
        clip=true]{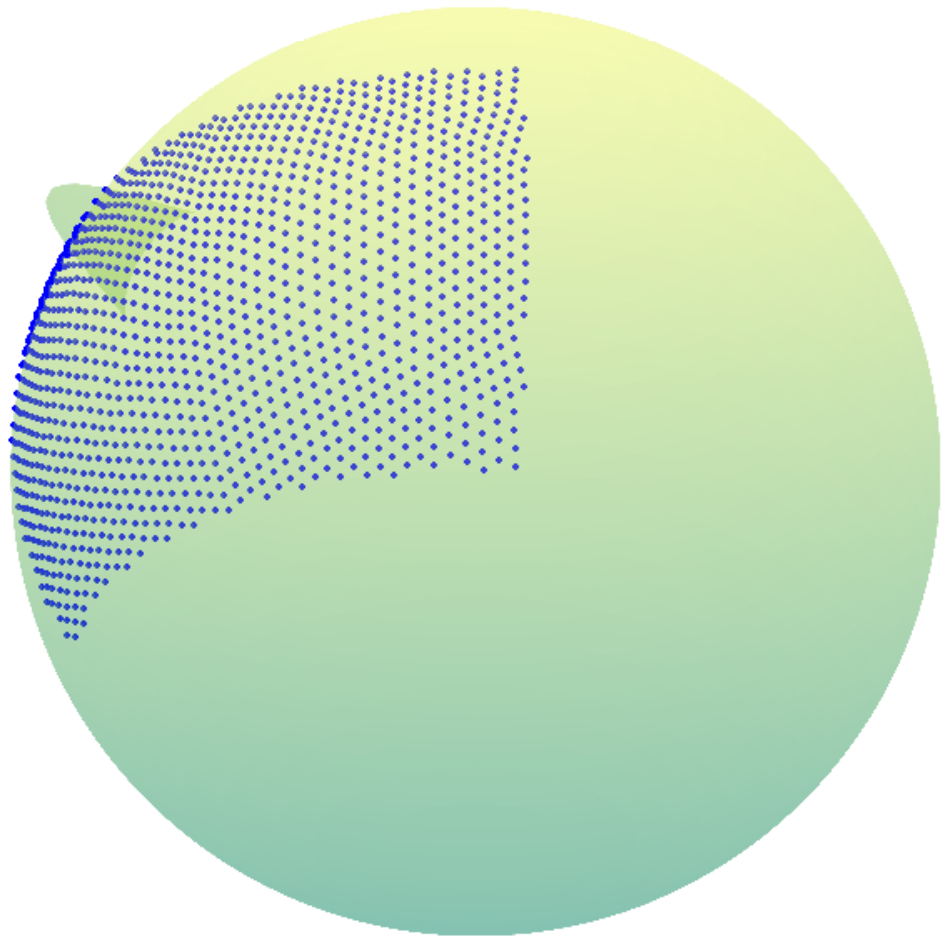}
    };

     \node at (-6.5,-3) { $W_{n-1}$};
     \node at ( -0.50,-3) { $W_{n-2}$};
     \node at ( 5.5,-3) { $W_{n-3}$};

    \end{tikzpicture}
    \caption{Support of MB functions corresponding to projection
      coefficients with absolute values greater than $10^{-4}$. The
      Gaussian perturbation $w(\theta,\varphi)$ is over imposed on the
      unit sphere. Notice that the support indicates the location of
      the detected signal $w(\phi,\varphi)$ from the
      $u(\phi,\varphi,\omega)$.}
    \label{mls:fig5}
\end{figure}

\section{Last Comments} In this paper we have developed a new approach for
change detection by applying the tools that are available to us from
functional analysis. By leveraging the power of the KL and other
tensor product expansions a multilevel nested functional spaces are
constructed. These spaces can be used to detect extraneous signals
that are orthogonal to the truncated eigenspace.

Our results show that this method is very flexible allowing the
application to complex domains. Furthermore, this approach can also
applied to other forms of tensor product expansions such as Polynomial
Chaos Expansions (PCE).  We have mostly shown examples of detecting a
fixed perturbation. However, this approach can be extended to fully
stochastic perturbations.  Future work involves collecting samples and
formulating hypothesis tests. In addition, we also envision
applications to machine learning classification.

\bigskip

\noindent 
\textbf{Acknowledgements:} I appreciate the help and advice from 
Trevor Martin. In particular, for proof reading the manuscript.

\bibliographystyle{plain}

\end{document}